\documentclass[11pt]{article}

\usepackage{anysize}

\marginsize{1.1in}{1.1in}{0.55in}{0.8in} 

\RequirePackage[OT1]{fontenc}
\RequirePackage{amsthm,amsmath,amssymb,amsfonts}
\RequirePackage[numbers]{natbib}
\RequirePackage[colorlinks,citecolor=blue,urlcolor=blue]{hyperref}

\usepackage[labelfont=bf]{caption}

\usepackage{amstext} 
\usepackage{array}   
\usepackage{enumitem} 
\setlist[itemize]{noitemsep} 
\allowdisplaybreaks
\usepackage{multirow}
\usepackage[flushleft]{threeparttable}
\usepackage{verbatim}
\usepackage{hyperref}
\usepackage{bm,bbm}
\usepackage{longtable}
\usepackage{rotating}
\usepackage{indentfirst}
\usepackage{subcaption}

\numberwithin{equation}{section}
\theoremstyle{plain}
\newtheorem{theorem}{Theorem}[section]
\newtheorem{lemma}{Lemma}[section]

\newtheorem{corollary}[theorem]{Corollary}

\theoremstyle{remark}\newtheorem{assumption}{Assumption}
\theoremstyle{remark}
\DeclareMathOperator*{\argmin}{arg\,min}

\newcommand{\DD}{\mathbb{D}}
\newcommand{\RR}{\mathbb{R}}

\newcommand{\ZZ}{\mathbb{Z}}
\newcommand{\LL}{\mathbb{L}}

\newcommand{\N}{\mathrm{N}}
\renewcommand{\P}{\mathrm{P}}
\newcommand{\E}{\mathrm{E}}
\renewcommand{\d}{\mathrm{d}}
\newcommand{\bj}{\bm{j}}
\newcommand{\bJ}{\bm{J}}

\newcommand{\by}{\bm{y}}
\newcommand{\bX}{\bm{X}}
\newcommand{\bx}{\bm{x}}
\newcommand{\bu}{\bm{u}}
\newcommand{\bv}{\bm{v}}
\newcommand{\bl}{\bm{l}}
\newcommand{\bmm}{\bm{m}}
\newcommand{\btheta}{\bm{\theta}}
\newcommand{\cF}{\mathcal{F}}
\newcommand{\cG}{\mathcal{G}}

\newcommand{\cL}{\mathcal{L}}
\newcommand{\Ind}{\mathbbm{1}}
\newcommand{\ceil}[1]{\left\lceil #1 \right\rceil}
\newcommand{\abs}[1]{\left| #1 \right|}

\newcommand{\trans}{{\mathrm{T}}}

\begin{document}

\title{Bayesian Inference for Multivariate Monotone Densities
\thanks{
The research was supported in part by NSF Grant number DMS-1916419.
}}

\author{Kang Wang \thanks{Department of Statistics, North Carolina State Universtiy. Email: kwang22@ncsu.edu}
\and
Subhashis Ghosal\thanks{Department of Statistics, North Carolina State Universtiy. Email: sghosal@stat.ncsu.edu}}

\date{}
\maketitle

\begin{abstract}
\noindent
We consider a nonparametric Bayesian approach to estimation and testing for a multivariate monotone density. Instead of following the conventional Bayesian route of putting a prior distribution complying with the monotonicity restriction, we put a prior on the step heights through binning and a Dirichlet distribution. An arbitrary piece-wise constant probability density is converted to a monotone one by a projection map, taking its $\mathbb{L}_1$-projection onto the space of monotone functions, which is subsequently normalized to integrate to one. We construct consistent Bayesian tests to test multivariate monotonicity of a probability density based on the $\mathbb{L}_1$-distance to the class of monotone functions. The test is shown to have a size going to zero and high power against alternatives sufficiently separated from the null hypothesis. To obtain a Bayesian credible interval for the value of the density function at an interior point with guaranteed asymptotic frequentist coverage, we consider a posterior quantile interval of an induced map transforming the function value to its value optimized over certain blocks. The limiting coverage is explicitly calculated and is seen to be higher than the credibility level used in the construction. By exploring the asymptotic relationship between the coverage and the credibility, we show that a desired asymptomatic coverage can be obtained exactly by starting with an appropriate credibility level. 
\vskip 0.3cm

\noindent
{\bf Keywords:} Bayesian tests for monotonicity;
Contraction rate;
Credible intervals;
Limiting coverage;
Multivariate density estimation.
\end{abstract}

\section{Introduction}\label{sec:introduction}

The problem of estimating a probability density from observed data has been well-studied theoretically, and the technique has been applied in diverse fields. Usually, density is estimated assuming the smoothness of the function, but in many contexts, structural properties like monotonicity, convexity, or log-concavity occur naturally. Under such a setting, a density function may be estimated accurately without even continuity or smoothness. The nonparametric maximum likelihood estimator of a univariate monotone nonincreasing density function, known as the Grenander estimator, is characterized as the left-derivative of the least concave majorant of the empirical distribution functions \cite{Grenander1956}. The pointwise distributional theory was developed in \cite{PrakasaRao1969,Groeneboom1985,Kulikov2005,durot2007,Durot2012}, among others. Applications in survival analysis were considered by \cite{huang1994estimating,huang1995hazard}. Testing for monotonicity was explored by \cite{hall2000testing,ghosal2000testing}, and others. Confidence bands and confidence intervals in monotone shape-restricted problems were constructed and studied in \cite{dumbgen2003optimal,dumbgen2004confidence,cai2013adaptive,Banerjee2001}; see \cite{groeneboom2014nonparametric} for a comprehensive summary. Accommodating monotonicity in the multivariate context and studying the properties of the resulting procedures are more challenging. 
Risk bounds of the least squares estimator of multivariate isotonic regressions given by a max-min formula were studied by \cite{chatterjee2015risk,han2019isotonic,han2021set}. A new estimator by using a subset of the blocks in the formula was proposed in \cite{Fokianos2020}. The risk bounds of this estimator were derived in \cite{deng2020isotonic} and \cite{han2020} obtained its asymptotic distribution. 

Bayesian approaches by incorporating the monotone shape-restriction in the prior for the regression function were considered in \cite{neelon2004bayesian,cai2007bayesian,shively2009bayesian}. The authors used spline functions to represent the regression function and put a prior on the coefficient vector subject to the monotonicity constraint. These works rely heavily on Markov chain Monte Carlo (MCMC) sampling techniques and did not provide any result about the concentration of the posterior distribution.  
 The posterior contraction rate for a univariate monotone density function using a Dirichlet process mixture of a uniform prior was derived by \cite{Salomond2014monotonicity}, utilizing 
the fact that a monotone nonincreasing density on the positive real half-line can be represented as a scale-mixture of the uniform kernel. The same representation 
to study the contraction rate of an empirical Bayes procedure was also used by  \cite{martin2019empirical}. A different approach was pursued by 
\cite{lin2014bayesian}. They did not incorporate the monotonicity restriction on the function at the prior stage and used a Gaussian process prior. Instead, they projected posterior samples on the space of monotone functions to comply with the shape restriction. The induced distribution, to be called the ``projection-posterior distribution'' is then used to make an inference on the regression function. The projection technique was also used by \cite{chakraborty2020coverage, chakraborty2021convergence}  starting from a prior on step-functions with step-heights given independent normal prior distributions. These papers respectively obtained global posterior contraction rates and asymptotic frequentist coverage of a credible interval of the function value at an interior point. The first paper also included a Bayesian test for  monotonicity with desirable asymptotic properties. For a monotone decreasing density, the projection-posterior approach through a step function having a Dirichlet prior distribution was considered by \cite{chakraborty2022density}. They obtained the optimal global posterior contraction rate, constructed a consistent Bayesian test for monotonicity, and characterized the limiting coverage of a Bayesian credible interval for the value of the density at an interior point. For both monotone regression and monotone density, it was observed that the limiting coverage of the credible interval based on projection-posterior quantiles is higher than the credibility level, but a predetermined asymptotic coverage level may be assured by starting with a slightly lower appropriate credibility level. The idea of a projection-posterior for inference is useful beyond the setting of shape-restricted curve estimation. In differential equation models, \cite{bhaumik2015bayesian,bhaumik2017bayesian,bhaumik2017efficient,bhaumik2022two} used  certain distance minimizing maps from the space of smooth functions to the space of the solutions of the differential equation to induce a posterior distribution on the parameters of the differential equation, starting from a random series prior on the regression function using a B-spline basis. They obtained Bernstein-von Mises theorems which, in particular, show that the posterior contracts at the parametric rate and Bayesian credible intervals have the correct asymptotic frequentist coverage. The Bayesian multivariate monotone regression problem was treated by \cite{Kang_reg}  using the projection-posterior approach. They obtained the global posterior contraction rate and a consistent Bayesian test for monotonicity. In the same setting, frequentist coverage of a Bayesian credible interval of the function value at a point using a more general ``immersion-posterior'' approach was studied by \cite{Kang_Coverage}.  

In this paper, we develop a Bayesian method to make inferences on the joint probability density function  of several random variables. Instead of making a global smoothness assumption, we assume that  the density function decreases in each argument. We shall pursue the projection-posterior approach, or use a more general immersion map to induce a posterior distribution from an unrestricted posterior distribution supported on piece-wise constant functions on small hyperrectangles. More precisely, we consider the setting where we observe independent and identically distributed random samples from an unknown probability density supported on $[0,1]^d$ which is monotonically decreasing in each coordinate. To put a prior on the probability density function, we partition $[0,1]^d$ into hyperrectangles and let the density function have constant values on each hyperrectangle. However, instead of imposing the shape constraints, we first put a Dirichlet distribution on the vector of (normalized) step-heights of the piece-wise representation as the prior. By posterior conjugacy, the corresponding unrestricted posterior distribution also follows a Dirichlet distribution on the unit simplex. Since the unrestricted posterior for the density function is not supported within the space of multivariate monotone densities, we comply with this restriction by applying a transformation on the samples from the unrestricted posterior, which maps a density to a multivariate monotone density. 
Depending on the purpose of the inference, we use different transformations. To obtain a global posterior contraction rate with respect to the $\LL_1$-metric, we use an $\LL_1$-projection of a density function onto the space of monotone functions. For a sampled density from the posterior distribution, the monotone projection results in a nonnegative monotone step function, which need not be a density though. This deficiency is rectified by combining it with a renormalization map. We show that the resulting induced posterior from this immersion map obtained by an $\LL_1$-projection and renormalization contracts at the optimal rate provided that the number of bins is chosen optimally. We then use this result to construct a Bayesian test for multivariate monotonicity of the density. We reject the hypothesis of monotonicity if the posterior probability that the distance between the density following the unrestricted posterior and the space of multivariate monotone functions is smaller than an appropriate threshold given by the posterior contraction rate is smaller than a predetermined level. We show that the resulting Bayesian test has size approaching zero, power at any fixed alternative approaching one, and the last assertion holds even for smooth alternative densities approaching the null hypothesis of multivariate monotone functions if they are separated sufficiently well. We further present a variation of the test such that the required separation with the null hypothesis for the power to approach one adapts to the smoothness of the underlying density function. It may be noted that no test, Bayesian or frequentist, seems available to test for multivariate monotonicity of a density in the literature. Finally, we address the important problem of uncertainty quantification of the function value at a given point. In the Bayesian approach, a credible interval, typically based on posterior quantiles is used to quantify the uncertainty about a parameter, but its frequentist coverage may be significantly different. Bayesian credible intervals with asymptotically adequate frequentist coverage are highly desirable. We obtain the asymptotic coverage of a Bayesian credible interval with endpoints at certain quantiles of an immersion-posterior distribution induced by a min-max, or max-min, block operation. We observe that the asymptotic coverage is slightly higher than the credibility level used in the construction. This is similar to the phenomenon observed for the univariate case in \cite{chakraborty2022density}, except that the immersion map is different from the $\LL_2$-projection map. As in the univariate case, the limiting coverage depends only on the credibility level used for a given smoothness level of the density function at the point of interest, not on any nuisance parameter. Hence the correct asymptotic coverage can be achieved by starting with an appropriate lower credibility level. 

The rest of the paper is organized as follows. In Section~\ref{sec:setup}, we introduce the notations, prior, posterior, and the notion of the immersion posterior. We present results on the posterior contraction rate of the immersion posterior and construct a consistent Bayesian test for multivariate monotonicity in Section~\ref{sec:rates tests}. The asymptotic coverage of a credible interval based on an immersion posterior is obtained in Section~\ref{sec:pointwise}. Simulation studies to assess the accuracy of the proposed methods are given in Section~\ref{sec:simulation}. The proofs  are given in Section~\ref{sec:proof} and Appendix \ref{sec:appendix}. 

\section{Setup, prior and posterior}
\label{sec:setup}

{\it Notations}. 
Let bold letters stand for $d$-dimensional vectors throughout the paper, either deterministic or random. The transpose of $\bm{a}$ is denoted by $\bm{a}^\trans$. Let $\RR$, $\RR_{\ge 0}$, and $\RR_{> 0}$ be the sets of real numbers, nonnegative real numbers, and positive real numbers respectively and $\ZZ$, $\ZZ_{\ge 0}$, and $\ZZ_{> 0}$ be the sets of integers, nonnegative integers, and positive integers.
For $(\bj_1,\bj_2)\in \ZZ_{>0}^d\times\ZZ_{>0}^d$, let $[\bj_1:\bj_2] = \{\bj\in \ZZ_{>0}^d: j_{1,k}\le j_k \le j_{2,k}, \text{ for all } 1\le k \le d\}$. Let $\bm{1}$ stand for the $d$-dimensional all-one vector and $\bm{0}$ for the $d$-dimensional all-zero vector. We use the notations $a\vee b = \max\{a, b\}$ and $a\wedge b =\min\{a, b\}$ for $(a, b)\in\RR^2$. Let $\bm{a} \vee \bm{b} = (a_1\vee b_1,\ldots, a_d\vee b_d)^\trans$ and $\bm{a} \wedge \bm{b} = (a_1\wedge b_1,\ldots, a_d\wedge b_d)^\trans$. For $a\in\RR$, $\ceil{a}$ denotes the smallest integer no greater than $a$. $\ceil{\bm{a}}$ denotes $(\ceil{a_1},\ldots,\ceil{a_d})^\trans$.
For the multi-index notations, let $\bl ! = l_1!\cdots l_d!$, $\bx^{\bl}=x_1^{l_1}\cdots x_d^{l_d}$, and $\partial^{\bl}f(\bx) = (\partial^{l_1\cdots l_d}  / \partial x_1^{l_1} \cdots \partial x_d^{l_d})f(\bx)$ for $\bl\in \ZZ_{\ge 0}^d$, $\bx\in\RR^d$, and some function $f$ on $\RR^d$.

For two sequences $a_n$ and $b_n$, we use $a_n = O(b_n)$ and $a_n = o(b_n)$ when $a_n/b_n$ is bounded and $a_n/b_n\to 0$ and $O_{\P}(\cdot)$, $o_{\P}(\cdot)$ are their stochastic versions. The symbols, $\lesssim$, and $\gtrsim$, represent less/greater than or equal to up to a positive constant, while $\asymp$ means being equal in order. We write $a_n\ll b_n$ if $a_n/b_n \to 0$, and $a_n\sim b_n$ when $a_n/b_n\to 1$.

The indicator function of a set $A$ is denoted by $\Ind_A$. 
Let $\LL_p(A)$, $p\ge 1$, and  $\LL_{\infty}(A)$ be the spaces of Lebesgue $p$-integrable and bounded functions on $A$ respectively. For $1\le p \le \infty$,
let $\|\cdot\|_p$ be the $p$-norm and define $d_p(f, \mathcal{S}) = \inf\{\|f-s\|_p:s\in\mathcal{S}\}$, for $f\in \LL_p(A)$ and $\mathcal{S}\subseteq \LL_p(A)$. The notation $d_H(\cdot, \cdot)$ denotes the Hellinger distance between two probability densities. Let $K(p,q)$ and $V(p,q)$ be the Kullback-Leibler divergence and Kullback-Leibler variation from probability density $p$ to probability density $q$. For a semimetric space $(\mathcal{T}, d)$, let $\mathcal{N}(\epsilon, \mathcal{T}, d)$ denote the $\epsilon$-covering number of $\mathcal{T}$.

Let $\N(\mu, \sigma^2)$ denote the normal distribution with mean $\mu$ and variance $\sigma^2$. 
The binomial distribution with parameters $n$ and $p$ is denoted by $\text{Bin}(n,p)$.
The $k$-dimensional Dirichlet distribution with parameters $(\alpha_1,\ldots,\alpha_k)$ is denoted by $\text{Dir}(k; \alpha_1,\ldots,\alpha_k )$. We write $\text{Ga}(a,b)$ to denote the Gamma distribution with shape parameter $a$ and scale parameter $b$. 
Let $\mathcal{L}(\cdot)$ denote the law of a random element.
Equality in distribution is denoted by $=_d$.
Convergences in probability and in distribution are denoted by $\to_{\P}$ and $\rightsquigarrow$, respectively.

We aim to make inferences on the unknown multivariate probability density function $g$ based on the independent and identically distributed (i.i.d.) sample of size $n$.
Specifically we have $n$ $d$-dimensional i.i.d. observations $\bX_1, \ldots, \bX_n$ from an unknown probability distribution $G$ with probability density function $g$ supported on $[0,1]^d$.
We shall use $\DD_n$ to denote the data.
The probability density function $g$ is presumed to be nonincreasing with respect to each variable when fixing the other ones.
Specifically, we define a natural partial ordering on $\RR^d$ in the following way. We say $\bx\succeq\by$ if and only if $x_k\geq y_k$ for all $1\le k \le d$, for $\bx,\by\in \RR^d$ and similarly define $\bx\preceq\by$.
Then $g$ is supposed to be monotonically nonincreasing with respect to the natural partial ordering. We denote the set of multivariate nonincreasing probability density functions as follows,
$
    \cG^* =\{g:[0,1]^d \to \RR_{\geq 0}: g(\bx)\leq g(\by) \text{ if } \bx\succeq\by; \int g = 1 \}.
$
Let 
$\cF^* =\{f:[0,1]^d \to \RR: f(\bx)\leq f(\by) \text{ if } \bx\succeq\by \}$ stand for the set of multivariate nonincreasing functions. Let $\mathcal{H}(\alpha,L)$ stand for the set of H\"{o}lder continuous function class on $[0,1]^d$ with H\"{o}lder smoothness index $\alpha$ and H\"{o}lder constant $L$.

{\it Prior and posterior}. 
The class of piece-wise constant functions with an increasing number of hyperrectangles approximates the class of all integrable functions. Therefore, to put a prior on the density function $g$, we can consider a distribution on piece-wise constant functions. To this end, partition the domain $[0,1]^d$ by splitting the $k$th direction into $J_k$ equal length subintervals, $\{[0,1/J_k], (1/J_k, 2/J_k], \ldots, ((J_k - 1)/J_k,1]\}$, $k=1,\ldots,d$, resulting in $\prod_{k=1}^d J_k$ pieces in total. Let $I_{\bm{1}} = \prod_{k=1}^d [0, 1/J_k]$ and $I_{\bj} = \prod_{k=1}^d ((j_k-1)/J_k, j_k/J_k]$ for $\bj\in [\bm{1}:\bJ]\setminus \{\bm{1}\}$. Then a typical function in the support of the prior can be represented as 
$g = (\prod_{k=1}^d J_k)\sum_{\bj \in [\bm{1}:\bJ]} \theta_{\bj}\Ind_{I_{\bj}}$, 
with a suitable prior on the scaled vector of step-heights $\bm{\theta}:=(\theta_{\bj}: \bj \in [\bm{1}:\bJ])$. Let the set of multivariate functions taking constant value on $I_{\bj}$, $\bj \in [\bm{1}:\bJ] $, be denoted by $\cF_{\bJ}$, 
and the set of multivariate piece-wise constant probability densities by
$\cG_{\bJ}$. These sets of functions with the constraint of multivariate monotonicity are respectively $\cF_{\bJ}^*=\cF^*\cap \cF_{\bJ}$ and $\cG_{\bJ}^*=\cG^*\cap \cG_{\bJ}$. 
For most results in this paper, $J_k$, $k=1,\ldots,d$, can be taken to be deterministic with value dependent on $n$. 

A conjugate prior $\Pi$ on $g$ is given by a Dirichlet distribution on $\bm{\theta}$: for some array of positive numbers $(\alpha_{\bj}:\bj \in [\bm{1}:\bJ])$, 
\begin{align}
\label{formula:dirichletprior}
    (\theta_{\bj}: \bj \in [\bm{1}:\bJ] )\sim \text{Dir}(\prod_{k=1}^d J_k; (\alpha_{\bj}: \bj \in [\bm{1}:\bJ])),
\end{align}
leading to the posterior measure $\Pi(\cdot|\DD_n)$ given by 
\begin{align}
\label{formula:dirichletposterior}
(\theta_{\bj}: \bj \in [\bm{1}:\bJ] )|\DD_n \sim  \mathrm{Dir}(\prod_{k=1}^d J_k; (\alpha_{\bj}+N_{\bj}: \bj \in [\bm{1}:\bJ])), 
\end{align}
where $
N_{\bj} = \sum_{i=1}^n \Ind_{I_{\bj}}(\bX_i)$.

Clearly, neither the prior nor the posterior is supported within the desirable space $\mathcal{G}^*$ of multivariate monotone densities. To comply with the shape restriction, we make an inference based on the posterior distribution $\Pi_n^*:=\Pi_n \circ \iota^{-1}$ induced by an immersion map $\iota$ that transforms a density to a multivariate monotone density. The choice of the immersion map will depend on the application. More specifically, to obtain the posterior contraction rate, we use an immersion map a composition of the $\LL_1$-projection on the space of monotone functions $\mathcal{F}^*$ followed by a renormalization taking in $\mathcal{G}^*$. Indeed, since the immersion map is applied to a sample from the posterior distribution of $g$ which is supported within $\cG_{\bJ}$, it will be observed that the image under the immersion map belongs to $\cG^*_{\bJ}$. Thus the immersion posterior is supported within the space of piece-wise constant multivariate monotone densities. This property has a significant implication that the computation of the immersion posterior consists entirely of a finite-dimensional sampling and a finite-dimensional optimization algorithm. To study the asymptotic frequentist coverage of a Bayesian credible interval, we use a min-max or max-min block operation, to be precisely defined in Section~\ref{sec:pointwise}. In both cases, the choice of the immersion maps is motivated by the desired asymptotic properties.

\section{Contraction rates and testing for multivariate monotonicity}
\label{sec:rates tests}

We first study the posterior contraction rate of the immersion posterior induced by the map $\iota$  mapping $g\in\cG_{\bJ}$ to $g^*\in \cG^*_{\bJ}$ given by $g^*=\tilde g/\int \tilde g$, where $\tilde g\in \argmin \{  \|g-h\|_1 : h\in \cF^*\}$. We observe an important fact below that for any $p\ge 1$, the $\LL_p$-projection of a piece-wise constant function belongs to $\cF^*_{\bJ}$, and hence $\iota(g)\in \cG^*_{\bJ}$ for any $g$ sampled from the posterior.

\begin{lemma}
\label{piece}
For $p\ge 1$ and $s\in \LL_p([0,1]^d)$, let $s_{\bJ} = (\prod_{k=1}^d J_k)\sum_{\bj\in[\bm{1}:\bJ]} b_{\bj} \Ind_{I_{\bj}}$ where $b_{\bj} = \int_{I_{\bj}} s$. Then, 
\begin{itemize}
    \item[{\rm (i)}] for any $h \in \cF_{\bJ}$, $\|s_{\bJ} - h\|_p \le \|s - h\|_p$;
    \item[{\rm (ii)}] if $s\in\cF^*$, then $s_{\bJ} \in \cF_{\bJ}^*$;
    \item[{\rm (iii)}] if $s\in \cG^*$, then $s_{\bJ}\in \cG_{\bJ}^*$.
\end{itemize}
\end{lemma}

In view of Lemma \ref{piece}, to obtain the monotone $\LL_p$-projection of a piece-wise function $h=\sum_{\bj\in[\bm{1}:\bJ]} \theta_{\bj}\Ind_{I_{\bj}}$, it suffices to isotonize the coefficient vector  $\btheta=(\theta_{\bj}: \bj \in [\bm{1}:\bJ])$ by minimizing $ \sum_{\bj\in[\bm{1}:\bJ]} |c_{\bj} - \theta_{\bj}|^p$ over 
$(c_{\bj}: \bj\in[\bm{1}:\bJ]) \in \mathcal{C}_{\bJ}$, 
where
$\mathcal{C}_{\bJ}=\{(c_{\bj}: \bj\in[\bm{1}:\bJ]) :c_{\bj_1} \le c_{\bj_2}, \text{ if } \bj_1 \succeq \bj_2\}$, the closed and convex monotone cone. The solution to this isotonization problem exists by the convexity of the loss function and the convexity and closeness of the feasible set. 
It is not hard to see that when $\theta_{\bj}\ge 0$ for all $\bj\in[\bm{1}:\bJ]$, the nonnegativity of the isotonization is ensured. However, for $p\ne 2$, the condition $\sum_{\bj\in[\bm{1}:\bJ]} \theta_{\bj}=1$ does not automatically ensure that the isotonized coefficients add up to one. For the case $d=1$ and $p=2$, from the description of the Pool Adjacent Violation Algorithm (cf. \cite{ayerPAVA}), it is clear that the sum of the coefficients does not change, and hence the isotonization of a density $g\in \cG_{\bJ}$ is automatically a density. This can be shown to be true for the multivariate case $d\ge 2$ and $p=2$ and thus makes the renormalization step redundant. Nevertheless, we work with a monotone $\LL_1$-projection as the approximation rate at a monotone function by piece-wise constant functions is not sufficiently strong in terms of the $\LL_p$-metric ($p>1$), and hence a renormalization step following the $\LL_1$-isotonization of the step-heights is necessary to make the monotone projection a valid probability density. An alternative would be to directly consider an $\LL_1$-projection of $g\in \cG_{\bJ}$ onto $\cG^*$, but a convenient computational algorithm for that does not seem to be readily available. This necessitates the use of a more general immersion map instead of a projection used for the multivariate monotone regression problem studied in \cite{Kang_reg}.    

Let $\tilde\btheta=(\tilde{\theta}_{\bj}: \bj\in[\bm{1}:\bJ])$ stand for the isotonization of $ \btheta$ with respect to the $\LL_1$-distance. There may be more than one solution, in which case we may choose any of them. Then the immersion posterior sample will be given by $\theta_{\bj}^{\ast} = \tilde{\theta}_{\bj}/\sum_{\bj\in[\bm{1}:\bJ]}{\tilde{\theta}_{\bj}}$ for all $ \bj\in[\bm{1}:\bJ]$. 
For the rest of the section, for all $k=1,\ldots,d$, we let $J_k=J$ for some $J$.  
The following result gives the posterior contraction rate of the resulting immersion posterior.

\begin{theorem}
\label{thm:contraction} 
Let the true density $g_0\in\cG^*$. If $1\ll J^d \ll n$ and $a_0\le \alpha_{\bj} \le a_1$
for all $\bj\in[\bm{1}:\bJ]$ and some $a_0,a_1>0$, then $\E_0\Pi(\|g^{\ast}-g_0\|_1>M_n \epsilon_n|\DD_n) \to 0$ for any $M_n \to \infty$, where $\epsilon_n=\max\{J^{-1}, \sqrt{J^d/n}\}$.
The optimal rate $n^{-1/(2+d)}$ is achieved by choosing $J\asymp n^{1/(2+d)}$.
\end{theorem}

Next, we construct a Bayesian test for multivariate monotonicity analogous to that in \cite{Kang_reg} for the multivariate monotone regression problem.

\begin{theorem}
\label{thm:test}
Let $\phi_n=\Ind \{\Pi(d_1(g, \cF^*) \leq M_n n^{-1/(2+d)}|\DD_n) < \gamma\}$, where $M_n$ is a slowly growing sequence and $\gamma\in(0,1)$ is a predetermined constant. If $J\asymp n^{1/(2+d)}$, then we have
\begin{enumerate}
    \item[{\rm (i)}] for any fixed $g_0\in\cG^*$, $\E_0(\phi_n) \to 0$;
    \item[{\rm (ii)}] for any fixed $g_0$ not in the $\LL_1$-closure of $\cG^*$, $\E_0(1-\phi_n) \to 0$;
    \item[{\rm (iii)}] for any $\alpha\in (0,1]$ and any $L>0$, we have 
    $\sup\{ \E_0 (1-\phi_n): g_0 \in \mathcal{H}(\alpha,L), d_1(g_0, \cG^*)>\rho_n(\alpha)\} \to 0$, 
    where
    \begin{align*}
        \rho_n(\alpha) = \left\{
        \begin{matrix}
        C n^{-\alpha/(2+d)}, & \text{ for some $C>0$ if $\alpha < 1$}, \\
        CM_n n^{-1/(2+d)}, & \text{ for any $C>2$ if $\alpha = 1$}.
        \end{matrix}
        \right.
    \end{align*}
\end{enumerate}
\end{theorem}

The universal consistency against any fixed non-monotone alternative density is appealing. Nevertheless, 
the separation rate $n^{-\alpha/(2+d)}$ of the alternative density from the null hypothesis of multivariate monotonicity assumed above for power to approach one is slower than the optimal posterior contraction rate $n^{-\alpha/(2\alpha+d)}$ of H\"older functions of smoothness index $\alpha$. This is because the optimal posterior contraction rate for this class using piece-wise constant functions is obtained only by choosing  $J\asymp n^{1/(2\alpha+d)}$, while the choice $J\asymp n^{1/(2+d)}$ used to construct the test leads to the suboptimal contraction rate $n^{-\alpha/(2+d)}$. However, the assumed choice is essential to obtain the optimal posterior contraction rate at multivariate monotone functions since otherwise, at some monotone true density, the posterior will not concentrate within its $M_n n^{-1/(2+d)}$-neighborhood, and hence the size of the Bayesian test will tend to one. The problem can be rectified by also putting an appropriate prior on $J$ but at the expense of higher computational complexity. Therefore, fixed $J$ can not give rise to the optimal separation rate for some classes of smooth functions, but with a random $J$, the required separation rate adapts optimally with $\alpha$ within a logarithmic factor.

\begin{theorem}
\label{thm:testingadaptive}
Let $J$ be endowed with a prior $\pi$ such that
    $e^{-b_1J^d\log J} \le \pi(J) \le e^{-b_2 J^d \log J}$ 
for some $b_1$ and $b_2>0$, and suppose that $a_2 n^{-a_3} \le \alpha_{\bj} \le a_1$, for some constants $a_1, a_2$, and $a_3>0$. Consider the test $\phi_n = \Ind \{ \Pi(d_1(g, \cF^*) \leq M_0 \sqrt{(J^d\log n)/n} |\DD_n) < \gamma \}$ for testing the hypothesis of multivariate monotonicity of $g$, where $M_0$ is a sufficiently large constant and $\gamma\in (0,1)$ is predetermined. 
Then 
\begin{enumerate}
    \item[{\rm (i)}] for fixed $g_0\in\cG^*$ and $g_0$ bounded away from zero, $\E_0(\phi_n) \to 0$;
    \item[{\rm (ii)}] for fixed density $g_0$ bounded away from zero and not belonging to the $\LL_1$-closure of $\cG^*$, $\E_0(1-\phi_n) \to 0$; 
    \item[{\rm (iii)}] there exists a sufficiently large constant $C>0$ depending only on $\alpha$, $L$, and $d$, such that $\sup\{\E_0(1-\phi_n): g_0 \in \mathcal{H}(\alpha, L), g_0\ge l >0, d_1(g_0,\cG^*) > C(n/\log n)^{-\alpha/(2\alpha + d)}\} \to 0$.
\end{enumerate}
\end{theorem}

\section{Coverage of pointwise credible intervals}
\label{sec:pointwise}

To obtain a Bayesian credible interval for $g(\bx_0)$ at an interior point $\bx_0\in (0,1)^d$ with an asserted frequentist coverage, we use quantiles of immersion posteriors induced by the immersion maps via the max-min (or min-max) operation over blocks containing $\bx_0$. The immersion map is partly inspired by the operation used in the estimator of the monotone function value at a point proposed by \cite{Fokianos2020}. The asymptotic distribution of such block estimators was studied by  \cite{han2020} for the multivariate monotone regression problem. We drew on their method in the proof of the following theorem. 

Throughout this section, $\bJ$ is taken as deterministic and changing with the sample size. We allow $J_k$ to take different values adapting to the different local smoothness levels along each coordinate.
Let $\bj_0 = \ceil{\bx_0 \circ \bJ}$ so that $\bx_0 \in I_{\bj_0}$. Since $\bx_0$ is arbitrary so far, we can define an immersion map $\iota$ by the value of $\iota (g)$ at $\bx_0$ for any $g = (\prod_{k=1}^d J_k)\sum_{\bj \in [\bm{1}:\bJ]} \theta_{\bj}\Ind_{I_{\bj}}$
with $\btheta$ from the the unit simplex in $\RR^{\prod_{k=1}^d J_k}$. Specifically, we consider the min-max immersion map $\overline{\iota}$ defined by $\overline{\iota}(g)(\bx_0)=\theta_{\bj_0}^{\star}$, where 
\begin{align}
\label{minmaxmap}
    \theta_{\bj_0}^{\star} = \min_{\bj_1 \preceq \bj_0} \max_{\bj_0 \preceq \bj_2} \frac{\sum_{\bj \in[\bj_1:\bj_2]}\theta_{\bj}}{\prod_{k=1}^{d} (j_{2,k} - j_{1,k} + 1)}.
\end{align}
It is clear that $\theta_{\bj_0}^{\star}$ is uniquely defined although there can be multiple pairs of $\bj_2$ and $\bj_1$ maximizing and minimizing the average of ${\theta_{\bj}}$ in the last display.
It is not hard to see that $\btheta^\star:=(\theta_{\bj}^{\star}: \bj \in [\bm{1}:\bJ]) \in \mathcal{C}_{\bJ}$.
If we switch the order of maximization and minimization in the transformation in \eqref{minmaxmap}, we obtain the max-min immersion map $\underline{\iota}$ defined by $\underline{\iota}(g)(\bx_0)=\theta_{\bj_0}^{\dagger}$, where 
\begin{align}
\label{maxminmap}
    \theta^{\dagger}_{\bj} = \max_{\bj \preceq \bj_2} \min_{\bj_1 \preceq \bj}  \frac{\sum_{\bj\in[\bj_1:\bj_2]}\theta_{\bj}}{\prod_{k=1}^{d} (j_{2,k} - j_{1,k} + 1)}.
\end{align}
It again follows that $\btheta^\dagger:=(\theta_{\bj}^{\star}: \bj \in [\bm{1}:\bJ]) \in \mathcal{C}_{\bJ}$, 
but $\btheta^{\dagger}$ is possibly different from $\btheta^{\star}$. As both operations result in monotone outcomes, the average immersion map $(\overline{\iota}+\underline{\iota})/2$ will also map piece-wise constant functions to monotone piece-wise constant functions. For a given $\bx_0$, we shall obtain asymptotic frequentist coverage of Bayesian credible intervals of $g(\bx_0)$ based on the quantiles of its immersion posterior distribution using the immersion maps $\overline{\iota}$, $\underline{\iota}$ and $(\overline{\iota}+\underline{\iota})/2$. 

We shall make the following assumptions on the prior and the true probability density function.

\begin{assumption}
\label{assumption:priors}
The parameters of the prior distribution satisfy $\max_{\bj}\alpha_{\bj} \leq a_1$ for some positive constant $a_1$.
\end{assumption}

\begin{assumption}
\label{assumption:localsmoothness}
The true density $g_0$ is a coordinate-wise nonincreasing function on $[0,1]^d$. $g_0$ is continuously differentiable in a small neighborhood of $\bx_0 \in (0, 1)^d$. For every $1\le k \le d$, there exists some $\eta_k \in \ZZ_{>0}$ such that $\partial_k^{m_k}g_0(\bx_0) = 0$ for $m_k = 1, \ldots, \eta_k - 1$ and $\partial_k^{\eta_k}g_0(\bx_0) \neq 0$, and 
for any $t>0$, with $M_0 = \{\bmm\in \ZZ_{\ge 0}^d : \sum_{k=1}^d m_k/\eta_k \le 1\}$,  
\begin{align*}
    \sup\big\{\big|g_0(\bx) -\sum_{\bmm \in M_0} \frac{\partial^{\bmm}g_0(\bx_0)}{\bmm !}(\bx - \bx_0)^{\bmm}\big| : |\bx-\bx_0| \preceq t \bm{r}_n\big\} = o(\omega_n), 
\end{align*}
where 
$\omega_n\downarrow 0$ and $\bm{r}_n = (\omega_n^{1/\eta_1}, \ldots, \omega_n^{1/\eta_d})^\trans$. \end{assumption}

Assumption \ref{assumption:localsmoothness} is adapted from \cite{han2020} and some unique features, essential for the proof, of multivariate monotone functions follow from this assumption, see Lemma 1 of \cite{han2020}.
Assumption \ref{assumption:localsmoothness} holds generally when $g_0$ has smoothness of order $\max \{\eta_k: 1\le k\le d\}$ at $\bx_0$.

To state the theorems on limiting coverage, we introduce some stochastic processes. Let  $H_1(\bu, \bv)$ and $H_2(\bu, \bv)$ be two independent centered Gaussian processes
indexed by $(\bu,\bv)\in \RR^d_{\ge 0}\times \RR^d_{\ge 0}$
with the covariance structure  
\begin{align*}
    \text{Cov}(H_i(\bu,\bv),H_i(\bu',\bv')) = \prod_{k=1}^d (u_k\wedge u'_k + v_k\wedge v'_k),  
\end{align*}
for $(\bu,\bv),(\bu',\bv') \in \RR_{\geq 0}^d \times \RR_{\geq 0}^d$ and $i=1,2$. Then define a Gaussian process by the relation
\begin{align*}
    U(\bu, \bv)  = \frac{\sqrt{g_0(\bx_0)}H_1(\bu, \bv)}{\prod_{k=1}^d(u_k+v_k)} + \frac{\sqrt{g_0(\bx_0)}H_2(\bu, \bv)}{\prod_{k=1}^d(u_k+v_k)} + \sum_{\bm{m}\in M} \frac{\partial_k^{m_k} g_0(\bx_0)}{(\bm{m}+\bm{1})!}\prod_{k=1}^d\frac{v_k^{m_k+1} - (-u)^{m_k}}{u_k+v_k}, 
\end{align*}
where $M = \{\bm{m}\in \ZZ_{\ge 0}^d: \sum_{k=1}^d m_k/\eta_k = 1\}$. For each of the three immersion maps, let the images be denoted by $g^\star=\overline{\iota}(g)$, $g^\dagger =\underline{\iota}(g)$ and $\bar{g}=((\overline{\iota}+\underline{\iota})/2)(g)$.

\begin{theorem}
\label{thm:weaklimit}
Let Assumptions~\ref{assumption:priors} and \ref{assumption:localsmoothness} hold and 
$\omega_n = n ^ {-1/(2+\sum_{k=1}^d \eta_k^{-1})}$ and 
$\bm{r}_n = (\omega_n^{1/\eta_1}, \ldots,\omega_n^{1/\eta_d})^{\trans}$. If $J_k \gg r_{n,k}^{-1}$ and $\prod_{k=1}^d J_k \ll n\omega_n$, then for every $z\in\RR$, we have
\begin{align*}
    &\Pi(\omega_n^{-1}(g^{\star}(\bx_0)-g_0(\bx_0))\leq z| \DD_n)
    \rightsquigarrow
    \P\big(
    \inf_{\bu\succeq \bm{0}}\sup_{\bv \succeq \bm{0}}
    U(\bu, \bv) \leq z| H_1
    \big),\\
    &\Pi(\omega_n^{-1}(g^{\dagger}(\bx_0)-g_0(\bx_0))\leq z| \DD_n)
    \rightsquigarrow
    \P\big(
    \sup_{\bv \succeq \bm{0}}\inf_{\bu\succeq \bm{0}}
    U(\bu, \bv) \leq z| H_1
    \big),\\
    &\Pi(\omega_n^{-1}(\bar{g}(\bx_0)-g_0(\bx_0))\leq z| \DD_n)
    \rightsquigarrow
    \P\big( \tfrac12 (\inf_{\bu\succeq \bm{0}}\sup_{\bv \succeq \bm{0}}
    U(\bu, \bv) +
    \sup_{\bv \succeq \bm{0}}\inf_{\bu\succeq \bm{0}}
    U(\bu, \bv)) \leq z| H_1
    \big).
\end{align*}
\end{theorem} 

The corresponding immersion-posterior  $(1-\gamma)$-quantiles are 
\begin{align*}
    & Q^{(1)}_{n,\gamma}=\inf \{z: \Pi(g^{\star}(\bx_0)\leq z|\DD_n)\geq 1-\gamma\},\\
    & Q^{(2)}_{n,\gamma}=\inf \{z: \Pi(g^{\dagger}(\bx_0)\leq z|\DD_n)\geq 1-\gamma\},\\
    & Q^{(3)}_{n,\gamma}=\inf \{z: \Pi(\bar{g}(\bx_0)\leq z|\DD_n)\geq 1-\gamma\}.
\end{align*}
Theorem~\ref{thm:weaklimit} gives the coverage of the corresponding immersion-posterior credible intervals in terms of the distributions of the following random variables: 
\begin{align*}
    & Z_B^{(1)}   = \P (
    \inf_{\bu\succeq \bm{0}}\sup_{\bv \succeq \bm{0}}
    U(\bu, \bv) \leq 0 | H_1),\\
    & Z_B^{(2)}   = \P (
    \sup_{\bv \succeq \bm{0}}\inf_{\bu\succeq \bm{0}}
    U(\bu, \bv) \leq 0 | H_1), \\
    & Z_B^{(3)}   = \P (\tfrac12 (\inf_{\bu\succeq \bm{0}}\sup_{\bv \succeq \bm{0}}
    U(\bu, \bv)
    \sup_{\bv \succeq \bm{0}}\inf_{\bu\succeq \bm{0}}
    U(\bu, \bv)) \leq 0 | H_1).
\end{align*}

\begin{corollary}
\label{cor:coverage}
The asymptotic coverage of one-sided immersion-posterior credible intervals is given by 
\begin{align}
\label{eq:one sided}
    \P_0(g_0(\bx_0)\leq Q^{(l)}_{n,\gamma})  & \to \P(Z_B^{(l)} \leq 1-\gamma ), 
\end{align}
and that of two-sided immersion-posterior credible intervals are
\begin{align}
\label{eq:two sided}
    \P_0(g_0(\bx_0)\in [ Q^{(l)}_{n,(1-\gamma)/2}, Q^{(l)}_{n,\gamma/2}])  & \to \P(\gamma/2 \le Z_B^{(l)}\le 1-\gamma/2 ), 
\end{align}
for $l=1,2,3$. 
\end{corollary}

Thus the coverage depends only on the distributions of $Z_B^{(l)}$, $l=1,2,3$, which do not involve any model parameters. The distribution functions of $Z_B^{(l)}$ at a set of most commonly used points under several local smoothness levels of $g_0$ at $\bx_0$ were tabulated in  \cite{Kang_Coverage}. From their tables, we can conclude that if $\bm{\eta}=(1,1)$, the asymptotic coverage of the three $95\%$ one-sided credible intervals considered here are $96.6\%$, $97.5\%$, and $96.8\%$, respectively. To target a specific asymptotic frequentist coverage,  \cite{Kang_Coverage} tabulated values of the inverted distribution functions of $Z_B^{(1)}$, $Z_B^{(2)}$ and $Z_B^{(3)}$ at several points and back-calculated the credibility level necessary for certain standard confidence levels. For example, if we use the symmetrized immersion-posterior quantile intervals with the targeted coverage of $95\%$, the required two endpoints are $Q^{(3)}_{n,0.959}$ and $Q^{(3)}_{n,0.041}$, resulting in a shorter interval than the nominal $95\%$ credible interval.

\section{Numerical results}\label{sec:simulation}

In this section, we shall carry out two simulation studies to investigate the approach we proposed in previous sections in finite sample settings. We shall look at both the global performance with respect to the $\LL_1$-metric and the pointwise inference in terms of the frequentist coverage of posterior quantile-based credible interval.

\subsection{Global deviation of immersion posteriors}
\label{subsec:simul_global}

We consider the following four data-generating probability density functions throughout this section.
\begin{itemize}
    \item[1.]
    $g_1(x, y) = 9(1-x)^2(1-y)^2$,
    \item[2.]
    $g_2(x, y) = 2.25\sqrt{(1-x)(1-y)}$,
    \item[3.]
    $g_3(x, y) = 4[(1+e^{12(x - 0.5)})(1+e^{12(y - 0.5)})]^{-1}$,
    \item[4.]
    $g_4(x, y) = 4[(1+e^{4(x - 0.5)})(1+e^{4(y - 0.5)})]^{-1}$.
\end{itemize}
For each density function, we generate i.i.d. samples of sizes, $n = 500, 1000, 2000, 5000,$ and $10000$. 
For density functions $g_1$ and $g_2$, we obtain the data by independent sampling from Beta distributions 
for each of the two coordinates $x$ and $y$.
For density functions $g_3$ and $g_4$, we use the rejection sampling to generate the samples.
To implement the Bayesian procedure, we set $J = \ceil{2n^{1/4}}$, which is the rate-optimal choice for the $\LL_1$-posterior contraction rates according to Theorem \ref{thm:contraction}. All hyperparameters $\alpha_{\bj}$ in the Dirichlet priors are set to $1$ without any prior information. 
The unrestricted posterior is immediate by conjugacy.
Next, we calculate an $\LL_1$-projection of the unrestricted posterior samples onto the multivariate monotone decreasing function class. We implement the $\LL_1$-projection algorithm from Section 4 of \cite{Stout2013}. This algorithm can isotonize a sample of size $s$ within $O(s\log s)$ time for data on a $2$-dimensional grid. The immersion posterior sample of a multivariate probability density function can then be obtained by a renormalization step following the $\LL_1$-projection.

From the proof of Theorem \ref{thm:contraction}, we know that the immersion posterior inherits the same contraction rate as the unrestricted posterior. 
In addition to evaluating the proposed methods when the sample size is finite, we also compare the $\LL_1$-metric between the immersion posterior samples and the true data generating density function and the $\LL_1$-metric between the unrestricted posterior samples and the true data generating density function in the simulation studies.

For each data set, we generate $1000$ posterior samples and calculate $1000$ immersion posterior samples correspondingly. For each posterior density function, we numerically calculate the $\LL_1$-metric to the true probability density. We repeat this process for $100$ times.

We summarize our computation results in Table \ref{tab:L1metric}. We reported the average of the posterior mean and posterior standard deviation of $\LL_1$-metric in terms of unrestricted posterior and immersion posterior. In Table \ref{tab:L1metric}, the average of the posterior mean is denoted by $L_1$ and $L_1^{\ast}$ respectively for unrestricted and immersion posterior, and the average of posterior standard deviation is denoted by $SD$ and $SD^{\ast}$. Marked in parentheses are standard  deviations of posterior means and posterior standard deviations over $100$ replicates for each setting. 

From our simulation results, immersion posterior improves the performance of the conventional unrestricted posterior by leveraging the functional shape information at a later stage by a monotone mapping. The improvement is consistent over all functions concerned here and all sample sizes in terms of the $\LL_1$-metric and its posterior variation. 

\begin{table}[ht]
\caption{$\LL_1$-metric of unrestricted posterior and immersion posterior.}
\label{tab:L1metric}
\centering
\begin{tabular}{|c|r|cccc|}
  \hline
$g_0$ &  $n$ & $L_1$ & $SD$ & $L_1^{\ast}$ & $SD^{\ast}$ \\ 
  \hline \multirow{4}{*}{$g_1$}
&  500 & 0.397(0.015) & 0.029(1.091) & 0.264(0.011) & 0.022(0.002) \\ 
& 1000 & 0.337(0.011) & 0.021(0.720) & 0.220(0.008) & 0.016(0.001) \\ 
& 2000 & 0.280(0.009) & 0.015(0.483) & 0.182(0.006) & 0.011(0.001) \\ 
& 5000 & 0.217(0.005) & 0.010(0.284) & 0.143(0.003) & 0.007(0.000) \\ 
&10000 & 0.180(0.004) & 0.007(0.188) & 0.119(0.002) & 0.005(0.000) \\
  \hline \multirow{4}{*}{$g_2$}
&  500 & 0.429(0.014) & 0.028(0.785) & 0.185(0.012) & 0.026(0.004) \\ 
& 1000 & 0.375(0.012) & 0.021(0.566) & 0.156(0.011) & 0.019(0.003) \\ 
& 2000 & 0.320(0.007) & 0.015(0.422) & 0.130(0.007) & 0.013(0.002) \\ 
& 5000 & 0.254(0.006) & 0.010(0.260) & 0.104(0.004) & 0.008(0.001) \\ 
&10000 & 0.214(0.004) & 0.007(0.160) & 0.087(0.003) & 0.006(0.000) \\
  \hline \multirow{4}{*}{$g_3$}
&  500 & 0.393(0.014) & 0.030(1.120) & 0.256(0.015) & 0.031(0.003) \\ 
& 1000 & 0.336(0.012) & 0.022(0.603) & 0.208(0.012) & 0.022(0.002) \\ 
& 2000 & 0.277(0.007) & 0.016(0.381) & 0.166(0.008) & 0.015(0.001) \\ 
& 5000 & 0.213(0.006) & 0.010(0.289) & 0.125(0.005) & 0.010(0.001) \\ 
&10000 & 0.177(0.004) & 0.007(0.187) & 0.101(0.004) & 0.007(0.001) \\
  \hline \multirow{4}{*}{$g_4$}
&  500 & 0.418(0.017) & 0.029(0.986) & 0.199(0.011) & 0.024(0.002) \\ 
& 1000 & 0.364(0.009) & 0.021(0.546) & 0.168(0.009) & 0.017(0.002) \\ 
& 2000 & 0.311(0.008) & 0.015(0.377) & 0.142(0.006) & 0.013(0.001) \\ 
& 5000 & 0.244(0.005) & 0.010(0.252) & 0.112(0.004) & 0.008(0.001) \\ 
&10000 & 0.207(0.004) & 0.007(0.180) & 0.094(0.003) & 0.006(0.000) \\ 
   \hline
\end{tabular}
\end{table}

It may be noted that as no other test for multivariate monotone density, Bayesian or frequentist, exists in the literature, we do not compare our test with any other procedure.

\subsection{Coverage of credible intervals}
\label{subsec:simul_cover}

In this part, we conduct simulation studies to investigate the frequentist coverage of credible intervals based on the immersion posterior quantiles.
The data is generated from the same set of probability density functions considered in the last section.
We consider five different sample sizes, $n=500, 1000, 2000$, and $5000$. The point at which we will make an inference on the density is $\bx_0 = (0.5, 0.5)$. 
According to the sample size, we take $J = \ceil{n^{1/4}\sqrt{\log n}}$.
The hyperparameters in the Dirichlet priors are taken to be $1$.
All three immersion maps are considered. They are indicated as ``min-max", ``max-min", and ``ave" in the following tables.
We considered the two-sided credible interval for $g_0(\bx_0)$ with four credibilities, $0.99, 0.95, 0.90,$ and $ 0.80$. To recalibrate the credible interval to the right asymptotic coverage, we use the corresponding quantiles of $Z_B^{(3)}$ from Table 4 of   \cite{Kang_Coverage}. Note that the distribution function of $Z_B^{(3)}$ is symmetric about $1/2$. Thus it is possible to get both the upper and lower quantiles from the table.
For example, we use $0.990$ and $0.010$, instead of $0.995$ and $0.005$, immersion quantiles to construct the credible sets targeting the coverage $99\%$.

We repeated generating the data in each setup for $1000$ times and calculate the frequency of credible intervals including the true parameter.
The results for the function $g_1$ are summarized in Table \ref{tab:covlen.g1} -- Table \ref{tab:covlen.g4}.
When the sample size is larger, the recalibrated credible intervals behave very well, with coverage staying closer to the target while the credible intervals based on raw quantiles are more conservative. However, when the sample size is moderate, the performance of all methods varies among different probability density functions.
For instance, for the function $g_1$, as the function value at $\bx_0$ is relatively closer to $0$, the proposed methods result in poor coverage when $n=500$, but soon the coverage gets much better when $n=1000$. In all the tables, $C$ and $L$ denote the coverage and average length of credible intervals, rounded to two digits.

\begin{table}[ht]
\caption{Coverage and length of credible intervals for $g_1(\bx_0)$}
\label{tab:covlen.g1}
\centering
\begin{tabular}{|c|c|cc|cc|cc|}
\hline
\multirow{3}{*}{$n$} & \multirow{3}{*}{\shortstack{immersion\\maps}}  &\multicolumn{6}{c|}{credibility}\\
\cline{3-8}
&&\multicolumn{2}{c|}{0.99} & \multicolumn{2}{c|}{0.95} &\multicolumn{2}{c|}{0.90}\\ 
\cline{3-8}
&& C & L & C & L & C & L\\
\hline 
\multirow{4}{*}{500}
& min-max & 0.91 & 0.86(0.12) & 0.70 & 0.66(0.10) & 0.56 & 0.56(0.09) \\ 
  & max-min & 0.92 & 0.85(0.13) & 0.74 & 0.66(0.10) & 0.60 & 0.56(0.09) \\ 
  & average & 0.91 & 0.85(0.12) & 0.71 & 0.66(0.10) & 0.58 & 0.56(0.09) \\ 
  & adjusted & 0.85 & 0.83(0.12) & 0.61 & 0.63(0.10) & 0.48 & 0.53(0.09) \\ 
  \hline
  \multirow{4}{*}{1000}& min-max & 1.00 & 0.62(0.09) & 0.97 & 0.48(0.07) & 0.94 & 0.41(0.06) \\ 
  & max-min & 1.00 & 0.62(0.09) & 0.98 & 0.48(0.08) & 0.95 & 0.41(0.07) \\ 
  & average & 1.00 & 0.62(0.09) & 0.97 & 0.48(0.08) & 0.94 & 0.41(0.07) \\ 
  & adjusted & 0.99 & 0.60(0.09) & 0.95 & 0.46(0.07) & 0.90 & 0.39(0.06) \\ 
  \hline \multirow{4}{*}{2000}& min-max & 1.00 & 0.54(0.08) & 0.98 & 0.42(0.06) & 0.93 & 0.36(0.06) \\ 
  & max-min & 1.00 & 0.54(0.08) & 0.98 & 0.42(0.06) & 0.95 & 0.35(0.06) \\ 
  & average & 1.00 & 0.54(0.08) & 0.98 & 0.42(0.06) & 0.94 & 0.35(0.06) \\ 
  & adjusted & 0.99 & 0.52(0.08) & 0.95 & 0.40(0.06) & 0.89 & 0.34(0.05) \\  
  \hline \multirow{4}{*}{5000}& min-max & 1.00 & 0.44(0.06) & 0.98 & 0.35(0.05) & 0.93 & 0.29(0.04) \\ 
  & max-min & 1.00 & 0.44(0.06) & 0.98 & 0.34(0.05) & 0.95 & 0.29(0.04) \\ 
  & average & 1.00 & 0.44(0.06) & 0.98 & 0.34(0.05) & 0.94 & 0.29(0.04) \\ 
  & adjusted & 0.99 & 0.43(0.06) & 0.95 & 0.33(0.05) & 0.90 & 0.27(0.04) \\ 
\hline
\end{tabular}
\end{table}

\begin{table}[ht]
\caption{Coverage and length of credible intervals for $g_2(\bx_0)$}
\label{tab:covlen.g2}
\centering
\begin{tabular}{|c|c|cc|cc|cc|}
\hline
\multirow{3}{*}{$n$} & \multirow{3}{*}{\shortstack{immersion\\maps}}  &\multicolumn{6}{c|}{credibility}\\
\cline{3-8}
&&\multicolumn{2}{c|}{0.99} & \multicolumn{2}{c|}{0.95} &\multicolumn{2}{c|}{0.90}\\ 
\cline{3-8}
&& C & L & C & L & C & L\\
\hline 
\multirow{4}{*}{500}
& min-max & 1.00 & 0.59(0.08) & 0.99 & 0.45(0.07) & 0.97 & 0.38(0.06)\\  
& max-min & 1.00 & 0.59(0.08) & 0.99 & 0.45(0.07) & 0.97 & 0.38(0.06)\\ 
& average & 1.00 & 0.58(0.08) & 0.99 & 0.45(0.07) & 0.97 & 0.38(0.06)\\ 
& adjusted & 1.00 & 0.56(0.08) & 0.98 & 0.43(0.06) & 0.94 & 0.36(0.06)\\ 
  \hline
  \multirow{4}{*}{1000}& min-max & 0.99 & 0.50(0.07) & 0.98 & 0.39(0.06) & 0.94 & 0.33(0.05) \\ 
  & max-min & 0.99 & 0.50(0.07) & 0.96 & 0.39(0.06) & 0.92 & 0.33(0.05)\\
  & average & 0.99 & 0.50(0.07) & 0.97 & 0.39(0.06) & 0.92 & 0.33(0.05)\\
  & adjusted & 0.98 & 0.48(0.07) & 0.94 & 0.37(0.05) & 0.88 & 0.31(0.05)\\ 
  \hline \multirow{4}{*}{2000}& min-max & 0.99 & 0.43(0.06) & 0.98 & 0.34(0.05) & 0.94 & 0.29(0.04)\\ 
  & max-min & 0.99 & 0.43(0.06) & 0.96 & 0.34(0.05) & 0.91 & 0.28(0.04)\\ 
  & average & 0.99 & 0.43(0.06) & 0.97 & 0.33(0.05) & 0.93 & 0.28(0.04)\\
  & adjusted & 0.99 & 0.41(0.06) & 0.94 & 0.32(0.05) & 0.89 & 0.27(0.04)\\ 
  \hline \multirow{4}{*}{5000}& min-max & 0.99 & 0.36(0.05) & 0.97 & 0.28(0.04) & 0.92 & 0.23(0.03)\\
  & max-min & 0.99 & 0.36(0.05) & 0.95 & 0.28(0.04) & 0.91 & 0.23(0.03)\\ 
  & average & 0.99 & 0.35(0.05) & 0.96 & 0.27(0.04) & 0.92 & 0.23(0.03)\\
  & adjusted & 0.99 & 0.34(0.05) & 0.93 & 0.26(0.04) & 0.87 & 0.22(0.03)\\ 
\hline
\end{tabular}
\end{table}

\begin{table}[ht]
\caption{Coverage and length of credible intervals for $g_3(\bx_0)$}
\label{tab:covlen.g3}
\centering
\begin{tabular}{|c|c|cc|cc|cc|}
\hline
\multirow{3}{*}{$n$} & \multirow{3}{*}{\shortstack{immersion\\maps}}  &\multicolumn{6}{c|}{credibility}\\
\cline{3-8}
&&\multicolumn{2}{c|}{0.99} & \multicolumn{2}{c|}{0.95} &\multicolumn{2}{c|}{0.90}\\ 
\cline{3-8}
&& C & L & C & L & C & L\\
\hline 
\multirow{4}{*}{500}
& min-max & 0.95 & 1.27(0.16) & 0.85 & 1.00(0.14) & 0.74 & 0.85(0.12) \\ 
& max-min & 0.96 & 1.26(0.16) & 0.86 & 0.99(0.14) & 0.76 & 0.84(0.13) \\ 
& average & 0.96 & 1.26(0.16) & 0.86 & 0.99(0.14) & 0.74 & 0.84(0.13) \\ 
& adjusted & 0.93 & 1.22(0.16) & 0.78 & 0.95(0.14) & 0.67 & 0.80(0.12) \\ 
\hline 
\multirow{4}{*}{1000}
& min-max & 1.00 & 1.00(0.13) & 0.98 & 0.78(0.11) & 0.93 & 0.66(0.10) \\ 
& max-min & 1.00 & 0.99(0.13) & 0.97 & 0.78(0.11) & 0.93 & 0.66(0.10) \\ 
& average & 1.00 & 0.99(0.13) & 0.98 & 0.78(0.11) & 0.93 & 0.66(0.10) \\ 
& adjusted & 0.99 & 0.96(0.13) & 0.94 & 0.75(0.11) & 0.89 & 0.63(0.10) \\ 
\hline 
\multirow{4}{*}{2000}
& min-max & 1.00 & 0.89(0.11) & 0.97 & 0.70(0.10) & 0.94 & 0.59(0.09) \\ 
& max-min & 1.00 & 0.88(0.11) & 0.97 & 0.69(0.10) & 0.93 & 0.59(0.09) \\ 
& average & 1.00 & 0.88(0.11) & 0.97 & 0.69(0.10) & 0.93 & 0.59(0.09) \\ 
& adjusted & 0.99 & 0.86(0.11) & 0.94 & 0.66(0.09) & 0.90 & 0.56(0.08) \\ 
\hline 
\multirow{4}{*}{5000}
& min-max & 1.00 & 0.76(0.09) & 0.98 & 0.59(0.08) & 0.93 & 0.50(0.07) \\ 
& max-min & 1.00 & 0.75(0.09) & 0.97 & 0.59(0.08) & 0.92 & 0.50(0.07) \\ 
& average & 1.00 & 0.75(0.09) & 0.98 & 0.59(0.08) & 0.93 & 0.50(0.07) \\ 
& adjusted & 0.99 & 0.73(0.09) & 0.95 & 0.56(0.08) & 0.87 & 0.47(0.07) \\ 
\hline
\end{tabular}
\end{table}

\begin{table}[ht]
\caption{Coverage and length of credible intervals for $g_4(\bx_0)$}
\label{tab:covlen.g4}
\centering
\begin{tabular}{|c|c|cc|cc|cc|}
\hline
\multirow{3}{*}{$n$} & \multirow{3}{*}{\shortstack{immersion\\maps}}  &\multicolumn{6}{c|}{credibility}\\
\cline{3-8}
&&\multicolumn{2}{c|}{0.99} & \multicolumn{2}{c|}{0.95} &\multicolumn{2}{c|}{0.90} \\ 
\cline{3-8}
&& C & L & C & L & C & L \\
\hline 
\multirow{4}{*}{500}
& min-max & 0.99 & 0.76(0.10) & 0.96 & 0.59(0.09) & 0.92 & 0.50(0.08) \\ 
& max-min & 1.00 & 0.76(0.11) & 0.96 & 0.59(0.09) & 0.94 & 0.50(0.08) \\ 
& average & 1.00 & 0.76(0.10) & 0.96 & 0.59(0.09) & 0.93 & 0.50(0.08)  \\ 
& adjusted & 0.99 & 0.73(0.10) & 0.94 & 0.56(0.08) & 0.87 & 0.47(0.07) \\ \hline
  \multirow{4}{*}{1000}
& min-max & 1.00 & 0.63(0.09) & 0.98 & 0.49(0.07) & 0.94 & 0.42(0.06) \\ 
& max-min & 1.00 & 0.63(0.09) & 0.97 & 0.49(0.07) & 0.93 & 0.41(0.06) \\ 
& average & 1.00 & 0.63(0.09) & 0.97 & 0.49(0.07) & 0.94 & 0.41(0.06)  \\ 
& adjusted & 0.99 & 0.60(0.08) & 0.95 & 0.46(0.07) & 0.89 & 0.39(0.06)  \\
  \hline \multirow{4}{*}{2000}
& min-max & 1.00 & 0.55(0.07) & 0.98 & 0.43(0.06) & 0.95 & 0.36(0.05)  \\ 
& max-min & 1.00 & 0.55(0.07) & 0.97 & 0.42(0.06) & 0.94 & 0.36(0.06) \\ 
& average & 1.00 & 0.54(0.07) & 0.98 & 0.42(0.06) & 0.94 & 0.36(0.05)  \\ 
& adjusted & 1.00 & 0.52(0.07) & 0.96 & 0.40(0.06) & 0.90 & 0.34(0.05)  \\ 
  \hline \multirow{4}{*}{5000}
& min-max & 1.00 & 0.46(0.06) & 0.99 & 0.36(0.05) & 0.97 & 0.30(0.05)  \\ 
& max-min & 1.00 & 0.46(0.06) & 0.98 & 0.36(0.05) & 0.96 & 0.30(0.05)  \\ 
& average & 1.00 & 0.46(0.06) & 0.99 & 0.36(0.05) & 0.96 & 0.30(0.04)  \\ 
& adjusted & 1.00 & 0.44(0.06) & 0.97 & 0.34(0.05) & 0.93 & 0.28(0.04) \\ 

   \hline
\end{tabular}
\end{table}

\section{Proofs}
\label{sec:proof}

\begin{proof}[Proof of Lemma \ref{piece}]
For part (i), let $\bx \in I_{\bj}$. Then, as  $h$ is constant over $I_{\bj}$, by the definition of $s_{\bJ}$
\begin{align}
\label{jensen}
    \abs{s_{\bJ}(\bx) - h(\bx)}^p = \big|\big(\prod_{k=1}^d J_k\big) \int_{I_{\bj}}(s - h)\big|^p
    \le \big(\prod_{k=1}^d J_k \big) \int_{I_{\bj}}|s - h|^p,
\end{align}
by Jensen's inequality. By taking integral on both sides of \eqref{jensen} over $I_{\bj}$, we have $\int_{I_{\bj}} |s_{\bJ} - h|^p \leq \int_{I_{\bj}} |s - h|^p$. Part (i) now follows then by summing over all $\bj$ on both sides of the inequality.

For parts (ii) and (iii), if $s$ is a multivariate monotone nonincreasing function, then $\{b_{j}\}$ is a monotonically decreasing array in terms of the natural partial order on the indices since, by the translation of $s$,
\begin{align*}
    b_{\bj_1} = \int_{I_{\bj_1}}s(\bx)\d\bx \ge \int_{I_{\bj_1}} s(\bx + (\bj_2 - \bj_1)/\bJ) \d\bx = \int_{I_{\bj_2}} s(\bx)\d\bx = b_{\bj_2},
\end{align*}
when $\bj_2\succeq\bj_1$. Additionally, it is clear that if $s\ge 0$, $b_{\bj}\ge 0$ for all $\bj$, and $\sum_{\bj\in[\bm{1}:\bJ]}(\prod_{k=1}^d J_k)b_{\bj} = \int_{[0,1]^d} s$.
\end{proof}

The main ideas of the proofs of Theorems~\ref{thm:contraction}, \ref{thm:test} and \ref{thm:testingadaptive} are very similar to their univariate counterpart in \cite{chakraborty2022density}, but the associated bounds will have to be reestablished in the multidimensional case. For the sake of self-contentedness, we present the proofs in brief.

\begin{proof}[Proof of Theorem \ref{thm:contraction}]
Let $\tilde{g}$ be an $\LL_1$-projection of $g$ to $\cF^*$, from which we obtain the normalized $g^{\ast}$. Since
$  \|g^{\ast} - g_0\|_1 \leq \|\tilde{g} - g_0\|_1 + \|\tilde{g} - g^{\ast}\|_1$,
and
\begin{align*}
    \|\tilde{g} - g^{\ast}\|_1 = \int\big|\tilde{g} - \frac{\tilde{g}}{\int \tilde{g}}\big| = \int \tilde{g} \big|1 - \frac{1}{\int \tilde{g}}\big| =  \big|\int \tilde{g} - \int g_0\big| \leq \|\tilde{g} - g_0\|_1,
\end{align*}
we have $\| g^{\ast} -  g_0\|_1 \le 2\| \tilde{g} - g_0\|_1$. Moreover, $\|\tilde{g} - g_0\|_1 \le
\|\tilde{g}-g\|_1 + \|g - g_{0}\|_1 \le 2\|g - g_{0}\|_1$ by the triangle inequality and the definition of projection. Combining $\| g^{\ast} -  g_0\|_1 \le 4\|g - g_0\|_1$, hence the immersion posterior contraction rate is inherited from the unrestricted posterior contraction rate.

Let $g_{0,J}=J^{d}\sum_{\bj\in[\bm{1}:\bJ]}\theta_{0,\bj} \Ind_{I_{\bj}}$, where $\theta_{0,\bj}=\P_0(I_{\bj})$.
For every $\bj$ and $\bx\in I_{\bj}$, $|g_{0,J}(\bx) - g_0(\bx)| = |J^d\int_{I_{\bj}} g_0 - g_0(\bx)| \le g_0 (({\bj}-\bm{1})/J) - g_0({\bj}/J)$ as $g_0$ is coordinate-wise decreasing. Then, splitting the integral over each $I_{\bj}$, 
\begin{align*}
    \|g_{0,J} - g_0\|_1 
    \le J^{-d}\sum_{\bj\in[\bm{1}:\bJ]} [g_0 (({\bj}-\bm{1})/J) - g_0({\bj}/J)]\le d J^{-1} (g_0(\bm{0}) - g_0(\bm{1}) ),
\end{align*}
by the telescoping property of the series corresponding to the points $\{\ldots, \bj, \bj+\bm{1},\bj+2\cdot\bm{1},\ldots\}$, and there are no more than $d J^{d-1}$ such series in the above summation.

Since $\|g - g_{0}\|_1 \le \|g - g_{0,J}\|_1 + \|g_{0,J} - g_{0}\|_1$, we only have to show that for every $M_n\to \infty$, 
\begin{align} 
\label{eq:posterior variation}
\E_0 \Pi (\|g-g_{0,J}\|_1\ge M_n \sqrt{J^d/n}|\DD_n) \to 0. 
\end{align}
As the $\LL_2$-norm dominates the $\LL_1$-norm, it suffices to show that $\E_0 \Pi (\|g-g_{0,J}\|_2\ge M_n \sqrt{J^d/n}|\DD_n)\to 0$. Observing that 
$\|g-g_{0,J}\|_2^2= J^{d} \sum_{\bj\in[\bm{1}:\bJ]} |\theta_{\bj}-\theta_{0,\bj}|^2$, it is enough to verify the bounds  
\begin{align*}
\sum_{\bj\in[\bm{1}:\bJ]} \mathrm{Var} (\theta_{\bj}|\DD_n)=O_P(n^{-1}), \quad 
\sum_{\bj\in[\bm{1}:\bJ]} |\E (\theta_{\bj}|\DD_n)-\theta_{0,\bj}|^2=O_P(n^{-1})
\end{align*}
in view of the Markov inequality and a standard variance-bias decomposition. Since $\text{Var}(\theta_{\bj}|\DD_n) \leq (\alpha_{\bj} + N_{\bj})/(\alpha_\cdot + n)^2$, the first part of the assertion is immediately obtained as $\alpha_{\bj}\le a_1<\infty$, $J^d\le n$ and $\sum_{\bj\in[\bm{1}:\bJ]}N_{\bj}=n$. For the second claim, note that $\E(\theta_{\bj}|\DD_n) = (\alpha_{\bj} + N_{\bj})/(\alpha_\cdot + n)$ and $N_{\bj}\sim \text{Bin}(n,\theta_{0,\bj})$, where $\alpha_\cdot= \sum_{\bj\in[\bm{1}:\bJ]} \alpha_{\bj}\le J^d a_1$.
Thus 
\begin{align*}
     \text{Var}(\E(\theta_{\bj}|\DD_n)) & = n\theta_{0,\bj}(1-\theta_{0,\bj})/(\alpha_0 + n)^2 \le \theta_{0,\bj}/n\le J^{-d} g_0(\bm{0})/n, \\
     (\E_0(\E(\theta_{\bj}|\DD_n)) - \theta_{0,\bj})^2 & = (\alpha_{\bj} - \alpha_0 \theta_{0,\bj})^2 / (\alpha_0 + n)^2 \le 2a_1^2(1+g_0^2(\bm{0}))/n^2,
\end{align*}
as $\theta_{0,\bj} \le J^{-d} g_0(\bm{0})$. Now summing over $\bj\in [\bm{1}:\bJ]$, we obtain the result.
\end{proof}

\begin{proof}[Proof of Theorem \ref{thm:test}]
(i) Since $d_1(g,\cF) = \|g - \tilde{g}\|_1 \le \|g - g_0\|_1$, the first assertion follows from Theorem \ref{thm:contraction}.

(ii) We claim that $d_1(g_0,{\cF}^*)>0$. If not, for every $\eta > 0$, there exists $f\in \cF^*$ such that $\|f-g_0\|\le \eta/2$.  Then $f_+\in \cF^*$ and $\|f_+ - g_0\|_1 \le \|f - g_0\|_1$, where $f_+$ is the positive part of $f$. Further, $\bar{f} = f_+/\|f_+\|_1\in \cG^*$ and, since $\bar f-f_+=f_+ (1-\|f_+\|_1)/\|f_+\|_1$ and $\|g_0\|_1=1$,  
\begin{align*}
    \|\bar{f} - g_0\|_1 \le \|\bar{f} - f_+\|_1 + \|f_+ - g_0\|_1 
    \le \abs{\|g_0\|_1 - \|f_+\|_1} + \|f_+ - g_0\|_1 
\end{align*}
is bounded by $2\|f_+ - g_0\|_1\le 2\|f - g_0\|_1 \le \eta $, contradicting the assumption that $g_0$ does not belong to the closure of $\cG^*$. Thus $d_1(g_0,\cF^*)>0$. Hence by the triangle inequality and the definition of the $\LL_1$-projection,
\begin{align}
\label{proj ineq}
    d_1(g, \cF^*) \ge \|g_0 - \tilde{g}\|_1 - \|g - g_0\|_1 \ge d_1(g_0,\cF^*) - \|g - g_0\|_1.
\end{align}
To conclude the consistency of the test, it now suffices to show that the posterior distribution of $g$ is consistent at $g_0$ with respect to the $\LL_1$-metric. By the martingale convergence theorem, we have $\|g_0 - g_{0,J}\|_1\to 0$ as $J\to \infty$. Observing \eqref{eq:posterior variation} does not require monotonicity, posterior consistency at $g_0$ follows.

(iii) Uniformly over $g_0\in \mathcal{H}(\alpha, L)$, we have $\|g_0 - g_{0,J}\|_1 \lesssim J^{-\alpha}$. Combined with  \eqref{eq:posterior variation}, this gives the posterior contraction rate $\max\{J^{-\alpha}, \sqrt{J^d/n}\}$  at $g_0$, which is optimized to $n^{-\alpha/(2+d)}$ for the choice   $J\asymp n^{1/(2+d)}$. As in the proof of part (ii), we can conclude that $d_1(g_0, \cF^*) \ge \rho_n(\alpha)/2$ if $d_1(g_0, \cG^*) > \rho_n(\alpha)$. By \eqref{proj ineq},
    $\Pi(d_1(g, \cF^*) \le M_n n^{-1/(2+d)}|\DD_n)$ is bounded by 
    $\Pi(\|g-g_0\|_1\ge d_1(g_0,\cF^*) - M_n n^{-1/(2+d)}|\DD_n)$.     
When $\alpha<1$, $n^{-\alpha/(d+2)}\gg M_n n^{-1/(d+2)}$ and hence the bound goes to $0$ in $\P_0$-probability with the choice of $\rho_n(\alpha)=C n^{-\alpha/(2+d)}$ for any fixed $C>0$. When $\alpha=1$, the bound is at most $\Pi(\|g-g_0\|_1\ge (C/2 - 1)M_n n^{-1/(2+d)}|\DD_n)$, which converges to zero in $\P_0$-probability.
\end{proof}

\begin{proof}[Proof of Theorem \ref{thm:testingadaptive}]
(i) By the definition of the piece-wise constant approximation, $g_{0,J} \in \cG^*\subset \cF^*$ if $g_0\in \cG^*$. Then we have $d_1(g, \cF^*) \le \|g - g_{0,J}\|_1$ for every $J>0$. 
Then for part (i), we assert that 
\begin{align}
\label{expression1}
   \P_0 ( \Pi(\|g - g_{0,J} \|_1  > M_0 \sqrt{(J^d\log n) /n} | \DD_n) \le 1-\gamma)\to 1.
\end{align}
To this end, we shall show that for $J_n^d \lesssim n$,
\begin{itemize}
    \item[(a)] 
        $\Pi(\|g - g_{0,J} \|_1 > M_0 \sqrt{J^d \log n / n}, J \le J_n | \DD_n) \to_{\P_0} 0$; 
    \item[(b)] $\Pi(J > J_n | \DD_n) \to_{\P_0} 0$ for $J_n\asymp (n/\log n)^{1/(d+1)}$.
\end{itemize}
As the $\LL_2$-metric dominates the $\LL_1$-metric on $[0,1]^d$,
the posterior probability in (a) is bounded by 
$\sum_{J=1}^{J_n} A_{n,J}(M_0)\pi(J|\DD_n)$, 
where $A_{n,J}(M_0) = \Pi(\int|g - g_{0,J} |^2 > M_0^2 (J^d \log n) / n| J, \DD_n)$.
By Markov's inequality 
\begin{align}
\label{An}
    A_{n,J}(M_0) \le  \frac{n}{M_0^2 J^d\log n}\big[ \sum_{\bj\in[\bm{1}:\bJ]} \mathrm{Var}(\theta_{\bj}|\DD_n)+  \sum_{\bj\in[\bm{1}:\bJ]}(\E(\theta_{\bj}|\DD_n) - \theta_{0,\bj})^2\big]. 
\end{align}
Using the bounds established in the proof of \eqref{eq:posterior variation}, the first term inside the bracket is $O(n^{-1})$ uniformly for all $J$, while we can decompose $\E(\theta_{\bj}|\DD_n) - \theta_{0,\bj}$ as the sum of $(\alpha_{\bj}-\alpha_\cdot N_{\bj}/n)/(\alpha_\cdot+n)$ and $N_{\bj}/n -\theta_{0,\bj}$ with $\alpha_\cdot:= \sum_{\bj\in[\bm{1}:\bJ]}\alpha_{\bj}\le a_1 J^d$. The sum of squares of the first term over $\bj\in[\bm{1}:\bJ]$ is bounded by 
\begin{align*}
    \frac{2\big[\sum_{\bj\in[\bm{1}:\bJ]} \alpha_{\bj}^2 + \alpha_\cdot^2  \sum_{\bj\in[\bm{1}:\bJ]} N_{\bj}^2/n^2\big]}{(\alpha_\cdot+n)^2 }  
    \le \frac{2[a_1^2 J^d+(a_1 J^d)^2 n \sum_{\bj\in[\bm{1}:\bJ]} N_{\bj}/n^2]}{(\alpha_\cdot+n)^2 } \lesssim \frac{J^d}{n}
\end{align*}
uniformly for all $\bj$ and all $J$. 
Using Bennett's inequality (cf. Proposition~A.6.2 of \cite{van1996weak}) and summing over $\bj$, we conclude that simultaneously for all $J\le J_n$, 
$	\max\{ |N_j/n-\theta_{0j}|: \bj\in[\bm{1}:\bJ]  \} \lesssim \sqrt{(\log n)/(nJ^d)}$ 
if $J_n^d \log n\lesssim n$. This gives that the expression in \eqref{An} is bounded by a constant multiple of $M_0^{-2}$, and hence can be made arbitrarily small by choosing $M_0$ sufficiently large. 
Thus, claim (a) follows.

To establish (b), we apply the general theory of posterior contraction in   \cite{ghosal2017fundamentals}. Because $g_0$ is bounded below, any $g$ sufficiently uniformly close is also bounded below, and so $g_0/g$ is uniformly bounded. Using standard relations between Kullback-Leibler divergences, the Hellinger distance and the $\LL_1$-distance, and recalling that the piece-wise constant approximation $g_{0,J}$ is $O(J^{-1})$ close to $g_0$ in $\LL_1$, to estimate the prior $\epsilon_n$-ball probability necessary for the application of the general theory, it suffices to lower bound $\Pi(J=J_n^*) \Pi ( \sum_{\bj\in[\bm{1}:\bJ_n^*] } |\theta_{\bj}-\theta_{0,\bj}|\le \epsilon_n) $, where $J_n^* \asymp \epsilon_n^{-1}$. By the assumption on the prior on $J$ and the Dirichlet small ball probability estimate (cf. Lemma~G.13 of  \cite{ghosal2017fundamentals}), a lower bound is $e^{-c (J_n^*)^d \log n}$ for some $c>0$ provided that $\epsilon_n$ is lower bounded by some negative power of $n$. Thus (b) holds if $J_n$ is chosen a sufficiently large constant multiple of $J_n^*$ by Theorem 8.20 of \cite{ghosal2017fundamentals}. Now, with these estimates, proceeding as in the proof of Theorem~\ref{thm:test}, the conclusion in (i) can be established.

For part (ii), in view of the martingale convergence theorem, for any given $\epsilon>0$, $\|g_0 - g_{0,J}\|_1 <\epsilon$ when $J\ge J_0$, say. Then we proceed as in part (i) with a fixed small $\epsilon$, $J_n^*=J_0$ and establish that (a) and (b) hold for $J_n$ a  multiple of $(n/\log n)^{1/d}$, where the constant of proportionality is taken to be sufficiently small depending on $\epsilon$.  

The proof of part (iii) also proceeds similarly after establishing the posterior contraction rate $\epsilon_n\asymp (n/\log n)^{-\alpha/(2\alpha+d)}$ by standard arguments for H\"older functions. The estimates used in the proof imply that (a) and (b) hold for $J_n$ a sufficiently large constant multiple of $(n/\log n)^{1/(2\alpha+d)}$. 
\end{proof}

The proof of Theorem~\ref{thm:weaklimit} is long, so we separate it into several key steps. These are stated as separate lemmas below, under the setup and assumptions of Theorem~\ref{thm:weaklimit}. We first introduce some notations and preliminaries. 

For $\bm{t}\in \RR^d$, let $j(\bm{t})=\ceil{(\bx_0 + \bm{t}\circ \bm{r}_n)\circ \bJ}$.
Then we can rewrite $g^{\star}(\bx_0)$ as 
\begin{align*}
    g^{\star}(\bx_0) = \min_{\bu\succeq \bm{0}} \max_{\bv\succeq \bm{0}} \frac{\sum_{\bj_0\in[j(-\bu): j(\bv)]} \theta_{\bj}}{\prod_{k=1}^d (j(\bv)_k-j(-\bu)_k+1)J_k^{-1}},
\end{align*}
and the localized version of $g^{\star}(\bx_0)$ with some positive constant $c$ and $\gamma$ as 
$$ g^{\star}_c(\bx_0) = \min_{c^{-\gamma}\bm{1}\preceq\bm{u}\preceq c\bm{1}}\max_{c^{-\gamma}\bm{1}\preceq\bm{v}\preceq c\bm{1}} \frac{\sum_{\bj_0\in[j(-\bu): j(\bv)]} \theta_{\bj}}{\prod_{k=1}^d (j(\bv)_k-j(-\bu)_k+1)J_k^{-1}}.$$
We denote
\begin{align*}
     W^{\star}_n &=\omega_n^{-1}(g^{\star}(\bx_0)-g_0(\bx_0)),\\  
    W^{\star}_{n,c} &=\omega_n^{-1}(g^{\star}_c(\bx_0)-g_0(\bx_0)),\\
    W_c &=\inf_{c^{-\gamma}\bm{1}\preceq\bm{u}\preceq c\bm{1}}\sup_{c^{-\gamma}\bm{1}\preceq\bm{v}\preceq c\bm{1}}\\
    & \Bigg\{
        \frac{\sqrt{g_0(\bx_0)} H_1(\bm{u},\bm{v})}{\prod_{k=1}^d(u_k+v_k)}+\frac{ \sqrt{g_0(\bx_0)} H_2(\bm{u},\bm{v})}{\prod_{k=1}^d(u_k+v_k)}+\sum_{\bm{m}\in M} \frac{\partial^{\bm{m}}g_0(x_0)}{(\bm{m}+\bm{1})!}
        \prod_{k=1}^d \frac{v_k^{m_k+1}-(-u_k)^{m_k+1}}{u_k+v_k}
        \Bigg\},\\
    W &= \inf_{\bm{u}\succeq \bm{0}}\sup_{\bm{v} \succeq\bm{0}} \\
    & \Bigg\{
        \frac{\sqrt{g_0(\bx_0)} H_1(\bm{u},\bm{v})}{\prod_{k=1}^d(u_k+v_k)}+\frac{ \sqrt{g_0(\bx_0)} H_2(\bm{u},\bm{v})}{\prod_{k=1}^d(u_k+v_k)} +\sum_{\bm{m}\in M} \frac{\partial^{\bm{m}}g_0(x_0)}{(\bm{m}+\bm{1})!}
        \prod_{k=1}^d \frac{v_k^{m_k+1}-(-u_k)^{m_k+1}}{u_k+v_k}
        \Bigg\}.
\end{align*}

The desired weak convergence of random measures on $\RR$ follows from Lemma B.1 of \cite{Kang_Coverage}, stated below.

\begin{lemma}[Lemma B.1 of \cite{Kang_Coverage}]
\label{prop:weakconvergence}
Let the following conditions be satisfied.
\begin{itemize}
\item[{\rm (i)}] For every $c>0$, $\cL(W^{\star}_{n,c}|\DD_n)\leadsto \cL(W_c|H_1)$;
\item[{\rm (ii)}] $\lim_{c\to\infty}\limsup_{n\to\infty} \Pi(W^{\star}_{n,c}\neq W^{\star}_n|\DD_n)=0$ in $\P_0$-probability;
\item[{\rm (iii)}] $\lim_{c\to\infty}\P(W_c\neq W)=0$.
\end{itemize}
Then in the space of probability measures under the weak topology, $\cL(W_n^{\star}|\DD_n)\leadsto \cL(W|H_1)$.
\end{lemma}

We apply the Gamma representation for the Dirichlet distribution to the unrestricted posterior $\btheta$ given the data $\DD_n$. Let $V_{\bj}|\DD_n \sim \text{Gamma} (\alpha_{\bj} + N_{\bj}, 1)$, mutually independent. Then $\theta_{\bj}$ given the data $\DD_n$ is distributed as $V_{\bj}/ \sum_{\bl\in[\bm{1}:\bJ]} V_{\bl}$. 
Let $\alpha_\cdot = \sum_{\bl\in[\bm{1}:\bJ]} \alpha_{\bl}$. 
It follows immediately that $\sum_{\bl\in[\bm{1}:\bJ]} V_{\bl} \sim \text{Gamma}(\alpha_\cdot + n, 1)$, and $\E(\sum_{\bl\in[\bm{1}:\bJ]}V_{\bl}) = \text{Var}(\sum_{\bl\in[\bm{1}:\bJ]} V_{\bl})=\alpha_\cdot + n$.


We decompose 
\begin{eqnarray*}
     \omega_n^{-1}(g^{\star}(\bx_0)-g_0(\bx_0))  =  \min_{\bu\succeq \bm{0}} \max_{\bv\succeq \bm{0}}\{A_{n,1}(\bu, \bv) + A_{n,2}(\bu, \bv) + B_{n, 1}(\bu, \bv) + B_{n, 2}(\bu, \bv)\},
\end{eqnarray*}
where
\begin{align*}
    A_{n,1}(\bu, \bv) & = 
    \omega_n^{-1}\frac{\sum_{[j(-\bu):j(\bv)]} (V_{\bj}- \E(V_{\bj}|\DD_n))/ \sum_{\bl}V_{\bl}}{\prod_{k=1}^d (j(\bv)_{k}-j(-\bu)_{k}+1)J_k^{-1}}, \\
    A_{n, 2}(\bu, \bv) & = \omega_n^{-1}\frac{\sum_{[j(-\bu):j(\bv)]} (\E(V_{\bj}|\DD_n)/\sum_{\bl}V_{\bl} - N_{\bj}/n)}{\prod_{k=1}^d (j(\bv)_{k}-j(-\bu)_{k}+1)J_k^{-1}},\\
    B_{n, 1}(\bu, \bv) & =\omega_n^{-1}\frac{\sum_{[j(-\bu):j(\bv)]} (N_{\bj} - \E_0(N_{\bj}))/n}{\prod_{k=1}^d (j(\bv)_{k}-j(-\bu)_{k}+1)J_k^{-1}},
    \\
    B_{n, 2}(\bu, \bv) & = \omega_n^{-1}\left(\frac{\sum_{[j(-\bu):j(\bv)]} \E_0(N_{\bj})/n}{\prod_{k=1}^d (j(\bv)_{k}-j(-\bu)_{k}+1)J_k^{-1}} - g_0(\bx_0)\right).
\end{align*}

\begin{lemma}\label{lem:Yn}
 The conditional distribution of stochastic process $Y_n(\bu,\bv;V)= \omega_n \sum_{[j(-\bu):j(\bv)]} (V_{\bj}- \E(V_{\bj}|\DD_n))$ given $\DD_n$ converges weakly to the distribution of $\sqrt{g_0(\bx_0)}H_2(\bu,\bv))$ in $\LL_{\infty}([\bm{0},c\bm{1}]\times[\bm{0},c\bm{1}])$
 in $\P_0$-probability.
\end{lemma}

\begin{proof}
Conditional on $\DD_n$, the expectation of $Y_n$ is zero and the variance is given by
\begin{align*}
    \text{Var}(Y_n(\bu,\bv;V)|\DD_n)=\omega_n^2 
    \sum_{\bj\in [j(-\bu):j(\bv)]} (\alpha_{\bj} + N_{\bj}).
\end{align*}
By Assumption 1, we have
\begin{multline*}
\omega_n^2\sum_{\bj\in [j(-\bu):j(\bv)]} \alpha_{\bj}\leq a_1 \cdot \omega_n^2\prod_{k=1}^d (r_{n,k}u_k + r_{n,k}v_k + 2J_k^{-1})J_{k} \\
= a_1 \cdot n^{-1} \prod_{k=1}^d J_k \prod_{k=1}^d (u_k+v_k+2r_{n,k}^{-1}J_k^{-1}),
\end{multline*}
as $\omega_n^2\prod_{k=1}^d r_{n,k} = n^{-1}$. By noting that $\prod_{k=1}^d J_k \ll n \omega_n \ll n$, $r_{n,k}^{-1} \ll J_k$, and $u_k, v_k \le c$ for every $1\le k \le d$, it follows that $\omega_n^2\sum_{[j(-\bu):j(\bv)]}\alpha_{\bj}\to 0$.

Since $g_0$ is differentiable, and hence is continuous at $\bx_0$,
\begin{align*}
\E_0\big(\omega_n^2 \sum_{\bj\in [j(-\bu):j(\bv)]} N_{\bj}\big) &= n\omega_n^2 \int_{\cup\{I_{\bj}: \bj\in [j(-\bu):j(\bv)]\} }g_0(\bx)\d \bx \\
 &=  n\omega_n^2 (g_0(\bx_0)+o(1)) \prod_{k=1}^d (r_{n,k}u_k + r_{n,k}v_k + O(J_k^{-1})) \\
 & \to  g_0(\bx_0)\prod_{k=1}^d (u_k+v_k),
\end{align*}
as $\omega_n^2\prod_{k=1}^d r_{n,k} = n^{-1}$ and $r_{n,k}^{-1} \ll J_k$.
Since $\sum_{\bj\in [j(-\bu):j(\bv)]}N_{\bj}$ is binomially distributed, we similarly have
\begin{multline*}
\text{Var}\Big(\omega_n^2 \sum_{\bj\in [j(-\bu):j(\bv)]} N_{\bj}\Big)
 \leq n\omega_n^4 \int_{\cup\{I_{\bj}: \bj\in [j(-\bu):j(\bv)]\} }g_0(\bx)\d \bx   \\
 \lesssim \omega_n^2g_0(\bx_0)\prod_{k=1}^d \max\{u_k+v_k, 2J_k^{-1} r_{n,k}^{-1}\} \to 0.
\end{multline*}
Thus $\text{Var}(Y_n(\bu,\bv;V)|\DD_n)\to_{\P_0} g_0(\bx_0)\prod_{k=1}^d(u_k + v_k)$.

We shall apply Lyapunov's central limit theorem using bounds for the  fourth moment to show the marginal convergence. Using the moment bound  $\E((V_{\bj}- \E(V_{\bj}|\DD_n))^4|\DD_n) \lesssim (\alpha_{\bj}+N_{\bj})^2$, and observing that 
\begin{multline*}
\omega_n^4\sum_{\bj\in [j(-\bu):j(\bv)]}(\alpha_{\bj} + N_{\bj})^2 \lesssim \omega_n^4\sum_{\bj\in [j(-\bu):j(\bv)]} N_{\bj}^2  \\
    \lesssim \omega_n^4 \prod_{k=1}^d(u_k r_{n,k} + v_k r_{n,k})J_k \big(\frac{n}{\prod_{k=1}^d J_k}\big)^2  
    \lesssim \frac{n \omega^2_n}{\prod_{k=1}^d J_k} = \frac{\prod_{k=1}^d r_{n,k}^{-1}}{\prod_{k=1}^d J_k},
\end{multline*}
with $\P_0$-probability tending to one, as $r_{n,k}^{-1} \ll J_k$ and $u_k, v_k \le c$ for every $k$. Comparing this with the asymptotic variance, it is clear that Lyapunov's condition holds. Thus for a fixed $(\bu,\bv)$, the conditional distribution of $Y_n(\bu,\bv,V)$ given the data converges weakly to $\sqrt{g_0(\bx_0)\prod_{k=1}^d(u_k +v_k)}\N(0,1)$ in $\P_0$-probability.  
The conclusion easily extends to finite-dimensional joint distributions by observing that for $(\bu,\bv)$ and $(\bu',\bv')\in [\bm{0},c\bm{1}]\times[\bm{0},c\bm{1}]$,
\begin{multline*}
    \E(Y_n(\bu,\bv;V) \cdot Y_n(\bu',\bv';V)|\DD_n)
    \\
    =\omega^2_n \sum_{\bj \in [j(-\bu\wedge\bu'):j(\bv\wedge\bv')]} \text{Var}(V_j|\DD_n)\to   g_0(\bx_0)\prod_{k=1}^d( u_k\wedge u_k' + v_k\wedge v_k'),
\end{multline*}
in $\P_0$-probability.

We now show the asymptotic tightness of the conditional law of the process given by $Y_n(\bu,\bv; V)$ given the data in $\LL_{\infty}([\bm{0},c\bm{1}]\times[\bm{0},c\bm{1}])$. We apply Theorem 2.2.4 of  \cite{van1996weak} with the function $\psi: x\mapsto x^{4d+2}$, and use the metric $\rho((\bu,\bv),(\bu',\bv'))=\sqrt{\|\bu-\bu'\|+\|\bv-\bv'\|}$ on $[\bm{0},c\bm{1}]\times[\bm{0},c\bm{1}]$. We first calculate the $(4d+2)$-th moment of the increments, conditional on the data. By Lemma \ref{lemma:gammamoments}, we have that 
\begin{align}
   &\E(|Y_n(\bu,\bv;V)-Y_n(\bu',\bv';V)|^{4d+2}|\DD_n) \nonumber\\
 &=\omega_n^{4d+2}\E\big(\big(\sum_{\bj\in [j(-\bu):j(\bv)]\triangle[j(-\bu'):j(\bv')]} (V_{\bj}-\E(V_{\bj}|\DD_n))\big)^{4d+2} | \DD_n \big) \nonumber \\
 &\lesssim  \omega_n^{4d+2}\big(\sum_{\bj\in [j(-\bu):j(\bv)]\triangle[j(-\bu'):j(\bv')]} N_{\bj}\big)^{2d+1},
  \label{formula:stirling}
\end{align}
with $\P_0$-probability tending to one, by noting that $\min \{N_{\bj}: \bj \in [\bm{1}:\bJ]\}\to_{\P_0}\infty$ and $\{\alpha_{\bj}\}$ is uniformly bounded.
With $\P_0$-probability tending to one, we have 
\begin{align*}
    \sum_{\bj\in [j(-\bu):j(\bv)]\triangle[j(-\bu'):j(\bv')]} N_{\bj} 
    \lesssim & n\prod_{k=1}^d r_{n,k} \big(\prod_{k=1}^d (u_k\vee u_k' + v_k\vee v_k') - \prod_{k=1}^d (u_k\wedge u_k' + v_k\wedge v_k')\big)\\
     =  & \omega^{-2}_n \big(\prod_{k=1}^d (u_k\vee u_k' + v_k\vee v_k') - \prod_{k=1}^d (u_k\wedge u_k' + v_k\wedge v_k')\big).
\end{align*}
Thus by Lemma \ref{lemma:delta}, \eqref{formula:stirling} is bounded by a constant multiple of $(\|\bu-\bu'\|+\|\bv-\bv'\|)^{2d+1}$ with $\P_0$-probability tending to one.
The $\epsilon$-packing number of $[\bm{0},c\bm{1}]\times[\bm{0},c\bm{1}]$ with respect to the metric $\rho$ is bounded by a constant multiple of $\epsilon^{-4d}$.
Then by Theorem 2.2.4 of  \cite{van1996weak}, by taking $\eta = \delta^{(2d+1)/(4d+1)}$, 
\begin{multline*}
    \| \sup\{|Y_n(\bu,\bv;V)-Y_n(\bu',\bv';V)|: \rho((\bu,\bv),(\bu',\bv'))\leq \delta \}\|_{4d+2} 
    \\
    \lesssim \int_0^{\eta} (\epsilon^{-4d})^{1/(4d+2)} \d\epsilon + \delta (\eta^{-8d})^{1/(4d+2)},
\end{multline*}
which is of the order of  $\delta^{1/(4d+1)}$. 
Hence the process $Y_n(\bu, \bv;V)$ is asymptotically uniformly equicontinuous. 
This concludes the proof of asymptotic tightness. 
\end{proof}

\begin{lemma}\label{lemma:P_nweaklimit}
Let $P_n(\bu, \bv) = \omega_n\sum_{\bj\in [j(-\bu):j(\bv)]} (N_{\bj} - \E_0 (N_{\bj}))$ and $J_k \gg r_{n,k}^{-1}$.
Then $P_n(\bu, \bv) \rightsquigarrow \sqrt{g_0(\bx_0)}H_1(\bu, \bv)$ in $\LL_{\infty}([\bm{0}, c\bm{1}]\times [\bm{0},c\bm{1}])$.
\end{lemma}

\begin{proof}
Let $Z_{n,i}(\bu,\bv) = \omega_n \Ind_{ \bigcup \{I_{\bj}: \bj\in [j(-\bu):j(\bv)]\}}(\bX_i)$ for $i = 1,\ldots, n$.
We rewrite $P_n(\bu, \bv)= \omega_n \sum_{i = 1}^n Z_{n,i}(\bu,\bv) - \E_0 Z_{n,i}(\bu,\bv).$ 
First, we verify the finite-dimensional convergence.
For $(\bu,\bv)$ and $(\bu',\bv')\in [\bm{0},c\bm{1}]\times[\bm{0},c\bm{1}]$, 
\begin{align*}
    \text{Cov}(P_n(\bm{u},\bm{v}),P_n(\bm{u}', \bm{v}'))  
  = n\omega_n^2 \E_0 Z_{n, 1}(\bm{u}\wedge \bm{u}', \bm{v}\wedge \bm{v}')
  - n\omega^2_n \E_0 Z_{n, 1}(\bu, \bv)  \cdot \E_0 Z_{n, 1}(\bu', \bv').
\end{align*}
We observe that 
\begin{multline*}
n\omega_n^2\E_0 Z_{n, 1}(\bu, \bv) = n\omega_n^2 \int_{\bigcup\{I_{\bj}: \bj\in [j(-\bm{u}):j(\bm{v})]\}  } g_0(\bx)\d\bx  
   \\
   = (g_0(\bx_0)+o(1)) \prod_{k=1}^d (u_k + v_k + O((J_k r_{n,k})^{-1})) 
  \to g_0(\bx_0)\prod_{k=1}^d (u_k + v_k),
\end{multline*}
uniformly for all $(\bu, \bv )\in [\bm{0}, c\bm{1}]\times [\bm{0}, c\bm{1}]$ as $J_k\gg r_{n,k}^{-1}$. Thus 
\begin{align*}
    \text{Cov}(P_n(\bm{u},\bm{v}),P_n(\bm{u}', \bm{v}')) \to g_0(\bx_0)\prod_{k=1}^d(u_k\wedge u'_k + v_k\wedge v_k').
\end{align*}

Next, we apply Theorem 2.11.9 of \cite{van1996weak} to show asymptotic tightness.
Take $m_n = n$ and $\cF = \{ (\bu,\bv)\in [\bm{0}, c\bm{1}]\times [\bm{0},c\bm{1}] \}$. By noting that $\|Z_{n, i}\|_{\cF} = \omega_n \Ind_{\bigcup\{I_{\bj}: \bj\in [j(-c\bm{1}):j(c\bm{1})]\}}(\bX_i)$ converges to $0$, and hence for every $\eta>0$,
\begin{align*}
    \sum_{i=1}^n \E_0 \|Z_{n, i}\|^2_{\cF} \Ind\{\|Z_{n, i}\|_{\cF}>\eta\}  = 0
\end{align*}
for sufficiently large $n$, trivially verifying the first condition of their theorem. 

Let $\varrho((\bu,\bv),(\bu',\bv')) = \|\bu-\bu'\|+\|\bv-\bv'\|$. Then 
\begin{eqnarray*}
    \lefteqn{ \sum_{i=1}^n \E_0 (Z_{n, i}(\bu,\bv) - Z_{n, i}(\bu',\bv'))^2}\\
    &&=  n\omega_n^2 \E_0|Z_{n, 1}(\bu,\bv) - Z_{n, 1}(\bu',\bv')|= n\omega_n^2 \int_{\cup\{I_{\bj}:\bj\in [j(-\bu): j(\bv)]\triangle [j(-\bu'):j(\bv')]\}}  g_0(\bx) \d\bx \\
    && \lesssim  \prod_{k=1}^d (u_k + v_k + 2(J_k r_{n,k})^{-1}) + \prod_{k=1}^d (u'_k + v'_k + 2(J_k r_{n,k})^{-1})-2\prod_{k=1}^d (u_k\wedge u'_k + v_k\wedge v'_k)\\
    &&\lesssim  \varrho((\bu,\bv),(\bu',\bv'))+ \max_{1\le k\le d} (J_k r_{n,k})^{-1},
\end{eqnarray*}
where the last inequality follows from Lemma \ref{lemma:delta}. Hence,
$\sum_{i=1}^n \E_0 (Z_{n, i}(\bu,\bv) - Z_{n, i}(\bu',\bv'))^2 \to 0,$
as $\varrho((\bu,\bv),(\bu',\bv')) \to 0$ and $n\to\infty$.

Next, for any $\epsilon>0$, we construct a partition of $\cF$ as follows. Choose a $\delta>0$, to be determined later, which depends only on $\epsilon$. For the interval $[0,c]$, with equispaced grid points $0=s_0 < s_1 <\ldots < s_l = c$, the partition of $(0,c]^d\times (0,c]^d$ is given by $\big\{\prod_{k=1}^{d}(s_{t_k-1},s_{t_k}] \times \prod_{k=1}^{d}(s_{r_k-1},s_{r_k}]: t_k, r_k\in \{1,\ldots, l\}\big\}$. Then  $ \mathcal{F}\subset \bigcup_{\bm{t},\bm{r}\in\{1,\ldots,l\}^d} \mathcal{F}_{\bm{t},\bm{r}}$, where 
 $  \mathcal{F}_{\bm{t},\bm{r}}= \big\{
    (\bu,\bv) \in (0,c]^d\times (0,c]^d: s_{t_k-1}< u_k \leq s_{t_k}, s_{r_k-1}< v_k \leq s_{r_k}
    \big\}.$
Now 
\begin{align*}
   & \sum_{i=1}^{n}\E_0 \sup_{f,g\in \mathcal{F}_{\bm{t},\bm{r}}} |Z_{n, i}(f)-Z_{n, i}(g)|^2 \\
   & \leq n\omega^2_n\E_0\big(\Ind_{\bigcup\{I_{\bj}:\bj\in [j(-\bm{s}_{\bm{t}}):j(\bm{s}_{\bm{r}})]\} }(\bX_1)-\Ind_{\bigcup\{I_{\bj}: \bj\in  [j(-\bm{s}_{\bm{t}-\bm{1}}):j(\bm{s}_{\bm{r}-\bm{1}})] \}}(\bX_1)\big)^2\\
   & \lesssim \prod_{k=1}^d (s_{t_k}+s_{r_k}+2(J_k r_{n,k})^{-1})- \prod_{k=1}^d(s_{t_k-1}+s_{r_k-1})  \\
   & \lesssim  \delta + \max_{1\le k\le d} (J_k r_{n,k})^{-1},
\end{align*}
where the last inequality follows from Lemma \ref{lemma:delta}.
As $J_k\gg r_{n,k}^{-1}$ for $k=1,\ldots, d$,
the second term in the last display will be eventually smaller than a prescribed $\delta>0$.
Thus $\delta$ can be set to be a multiple of $\epsilon^2$ to meet the second condition. Meanwhile, the bracketing number with respect to the semimetric $\LL_2^n$ defined therein, $\mathcal{N}_{[\,]}(\epsilon,\mathcal{F}, L_2^n)\leq l^{2d}\leq (2c/\delta)^{2d}\lesssim \epsilon^{-4d}$.
Then there exists a positive constant $C$ such that
\begin{align*}
    \int_0^{\delta_n}\sqrt{\log \mathcal{N}_{[\,]}(\epsilon, \mathcal{F}, L_2^n)}\, \d\epsilon \leq \int_0^{\delta_n} \sqrt{\log (C\epsilon^{-4d})}\,\d\epsilon\to 0,
\end{align*}
as $\delta_n\to 0$. This verifies all three conditions in Theorem 2.11.9 of  \cite{van1996weak}, and hence concludes the weak convergence of $P_{n}$ in $\LL_{\infty}([\bm{0},c\bm{1}]\times [\bm{0},c\bm{1}])$. 
\end{proof}

\begin{lemma}\label{lemma:AsposP1}
    Under the conditions of Theorem \ref{thm:weaklimit}, for any $\tau > 0$, it holds that 
    \begin{align*}
        \E_0 \sup_{
        \substack{\bu\succeq \tau\bm{1},\\
        \bv\succeq \tau\bm{1}
        }}
        \left|\frac{\sum_{[j(-\bu):j(\bv)]}(N_{\bj}-\E_0(N_{\bj}))/n}{\prod_{k=1}^d (j(\bv)_{k}-j(-\bu)_{k}+1)J_k^{-1}}\right| \lesssim \omega_n.
    \end{align*}
\end{lemma}

\begin{proof}
For ease of notation, we assume that $\tau = 1$.
\begin{multline}
\label{eqn:marti}
    \E_0 \sup_{\bu\succeq \bm{1}, \bv\succeq\bm{1}}\left|\frac{\sum_{\bj\in [j(-\bu):j(\bv)]}(N_{\bj}-\E_0(N_{\bj}))/n}{\prod_{k=1}^d (j(\bv)_{k}-j(-\bu)_{k}+1)J_k^{-1}}\right|
    \\
    \leq \sum_{\substack{h_k\geq 0\\1\leq k\leq d}}\sum_{\substack{h'_k\geq 0\\1\leq k\leq d}}
    \E_0 \max_{\substack{2^{h_k}\leq u_k\leq2^{h_k+1}\\2^{h'_k}\leq v_k\leq2^{h'_k+1}\\1\leq k\leq d}}\left|\frac{\sum_{\bj\in[j(-\bu+\bm{1}):j(\bv-\bm{1})]} (N_{\bj} - \E_0(N_{\bj}))}{n\prod_{k=1}^d (j(\bv)_{k}-j(-\bu)_{k}+1)J_k^{-1}}\right|.
\end{multline}
Let $\bm{h} = (h_1, \ldots, h_d)^{\trans}$ and $\bm{h}' = (h_1', \ldots, h_d')^{\trans}$. First, we have
\begin{multline}
\label{denomi}
    \min_{\substack{2^{h_k}\leq u_k\leq2^{h_k+1}\\2^{h'_k}\leq v_k\leq2^{h'_k+1}\\1\leq k\leq d}}
    n\prod_{k=1}^d (j(\bv)_{k}-j(-\bu)_{k}+1)J_k^{-1} 
    \\
     =   n\prod_{k=1}^d (j(2^{\bm{h}'})_{k}-j(-2^{\bm{h}})_{k}+1)J_k^{-1} 
    \ge   n\prod_{k=1}^d [r_{n,k} (2^{h_k} + 2^{h_k'})].
\end{multline}
By Jensen's inequality, 
\begin{align}\label{eqn:submart}
    \E_0 \max_{\substack{2^{h_k}\leq u_k\leq2^{h_k+1}\\2^{h'_k}\leq v_k\leq2^{h'_k+1}\\1\leq k\leq d}}\Big|\sum_{[j(-\bu):j(\bv)]} (N_{\bj} - \E_0(N_{\bj}))\Big|  
    \le \Big[\E_0 \max_{\substack{2^{h_k}\leq u_k\leq2^{h_k+1}\\2^{h'_k}\leq v_k\leq2^{h'_k+1}\\1\leq k\leq d}}\Big|\sum_{[j(-\bu):j(\bv)]} (N_{\bj} - \E_0(N_{\bj}))\Big|^2\Big]^{1/2}
\end{align}
By Lemma A.5 of  \cite{Kang_Coverage}, the right-hand side of \eqref{eqn:submart} is further bounded, up to a constant multiple depending only on $d$, by
$
    \Big[\E_0 \big|\sum_{\bj\in [j(-2^{\bm{h}+\bm{1}}):j(2^{\bm{h}'+\bm{1}})]} (N_{\bj} - \E(N_{\bj}))\big|^2 \Big]^{1/2}.
$
As $\sum_{\bj\in [j(-2^{\bm{h}+\bm{1}}):j(2^{\bm{h}'+\bm{1}})]}N_{\bj} \sim \text{Bin}(n, \int_{\bigcup\{I_{\bj}:\bj\in [j(-2^{\bm{h}+\bm{1}}):j(2^{\bm{h}'+\bm{1}})]\}} g_0(\bx)\d\bx)$, we have
\begin{multline*} 
\Big[\E_0 \big|\sum_{[j(-2^{\bm{h}+\bm{1}}):j(2^{\bm{h}'+\bm{1}})]} (N_{\bj} - \E(N_{\bj}))\big|^2 \Big]^{1/2}  
\\
 \le  \Big[ n \int_{\bigcup\{I_{\bj}:\bj\in [j(-2^{\bm{h}+\bm{1}}):j(2^{\bm{h}'+\bm{1}})]\}} g_0(\bx)\d\bx
    \Big]^{1/2}  
\lesssim  \Big[n\prod_{k=1}^dr_{n,k}(2^{h_k+1}+2^{h'_k+1})\Big]^{1/2},
\end{multline*}
as $g_0$ is bounded and $J_k\gg r_{n,k}^{-1}$.
Combining with \eqref{denomi}, we conclude that \eqref{eqn:marti} is bounded, up to some positive constant, by
$
\omega_n \sum_{\substack{h_k\geq 0\\1\leq k\leq d}}\sum_{\substack{h'_k\geq 0\\1\leq k\leq d}} \prod_{k=1}^d(2^{h_k}+2^{h'_k})^{-1/2}.
$
Then it remains to show 
the convergence of the sum of the series in the last display.
To see this, 
using the fact that $\sum_{h\geq 0} \sum_{h'\geq 0}(2^h+2^{h'})^{-1/2}$ converges, we have 
\begin{align*}
    \sum_{\substack{0\le h_k\le n_k\\1\leq k\leq d}}\sum_{\substack{0\le h'_k\le n_k'\\1\leq k\leq d}} \prod_{k=1}^d(2^{h_k}+2^{h'_k})^{-1/2} 
    = 
    \prod_{k=1}^d
    \Big(\sum_{0\le h_k\le n_k} \sum_{0\le h'_k\le n'_k}(2^{h_k}+2^{h_k'})^{-1/2}\Big),
\end{align*}
and it will converge to some positive constant when $n_k$ and $n_k'$ go to infinity.
\end{proof}

\begin{lemma}\label{lemma:contractionrate}
Under the conditions of Theorem \ref{thm:weaklimit}, for any $M_n \uparrow \infty$, 
we have $\Pi(|g^{\star}(\bx_0) - g_0(\bx_0)|\geq M_n\omega_n| \DD_n) \to_{\P_0} 0$.
\end{lemma}

\begin{proof}
We shall prove one side of the claim that 
\begin{align}
    \P_0( \Pi(g^{\star}(\bx_0) - g_0(\bx_0)\geq M_n\omega_n|\DD_n) \geq  \eta) \to 0 
    \label{oneside}
\end{align} 
for any $\eta>0$. The other side follows by similar arguments.

By the min-max formula, we have that 
\begin{align*}
& \omega_n^{-1} (g^{\star}(\bx_0) - g_0(\bx_0) ) \\
   & \leq \omega_n^{-1} \left(\max_{\bv\succeq\bm{0}} \frac{\sum_{[j(-\bm{1}):j(\bv)]} \theta_{\bj}}{\prod_{k=1}^d (j(\bv)_{k}-j(-\bm{1})_{k}+1)J_k^{-1}} - {g_0(\bx_0)}\right)\\
& \leq \max_{\bv\succeq \bm{0}}| A_{n,1}(\bm{1},\bv;V)| + \max_{\bv\succeq \bm{0}}| A_{n, 2}(\bm{1},\bv;V)| + \max_{\bv\succeq \bm{0}}| B_{n,1}(\bm{1},\bv)| + \max_{\bv\succeq \bm{0}}|B_{n,2}(\bm{1},\bv)|.
\end{align*}
Thus it follows that 
\begin{align*}
    &\Pi(g^{\star}(\bx_0) - g_0(\bx_0)\geq M_n\omega_n|\DD_n)\\
    \leq &  \Pi\big(\max_{\bv\succeq \bm{0}}|A_{n,1}(\bm{1},\bv;V)|>M_n/4|\DD_n\big) + \Pi\big(\max_{\bv\succeq \bm{0}}|A_{n, 2}(\bm{1},\bv;V)|>M_n/4|\DD_n\big) \\
    & \quad +\Ind\big\{\max_{\bv\succeq \bm{0}}|B_{n,1}(\bm{1},\bv)|>M_n/4\big\} + \Ind\big\{\max_{\bv\succeq \bm{0}}|B_{n,2}(\bm{1},\bv)|> M_n/4\big\}.
\end{align*}
Hence, we have
\begin{eqnarray}
    &\lefteqn{ \P_0( \Pi(g^{\star}(\bx_0) - g_0(\bx_0)\geq M_n\omega_n|\DD_n) \leq \eta) }\nonumber\\
    &&\leq  \P_0\big( \Pi\big(\max_{\bv\succeq \bm{0}}|A_{n,1}(\bm{1},\bv;V)|>M_n/4\big)> \eta/2|\DD_n\big) \label{probAn1}\\
    &&  \quad + \P_0\big( \Pi\big(\max_{\bv\succeq \bm{0}}|A_{n, 2}(\bm{1},\bv;V)|>M_n/4|\DD_n\big)> \eta/2 \big) 
    \label{probAn2}\\
    && \quad + \P_0\big(\max_{\bv\succeq \bm{0}}|B_{n,1}(\bm{1},\bv)|>M_n/4 \big)\label{probBn1}\\
    && \quad + \P_0\big(\max_{\bv\succeq \bm{0}}|B_{n,2}(\bm{1},\bv)|> M_n/4  \big)
    \label{probBn2}.
\end{eqnarray}
It suffices to show that each term \eqref{probAn1}--\eqref{probBn2} converges to zero. 

To show that \eqref{probAn1} converges to zero, it is enough to show that 
\begin{align}
\label{expratebound}
    \E\Big(\max_{\bv\succeq \bm{0}}
    \big|\frac{\sum_{[j(-\bm{1}):j(\bv)]} (V_{\bj}- \E(V_{\bj}))/n}{\prod_{k=1}^d (j(\bv)_{k}-j(-\bm{1})_{k}+1)J_k^{-1}}\big|\Big| \DD_n\Big) = O_{\P_0}(\omega_n).
\end{align}
We use arguments similar to those in the proof of Lemma \ref{lemma:AsposP1}. By splitting the domain into smaller rectangles, we can bound 
\begin{align}
    \lefteqn{ \E\Big(\max_{\bv\succeq \bm{0}}
    \big|\frac{\sum_{[j(-\bm{1}):j(\bv)]} (V_{\bj}- \E(V_{\bj}))/n}{\prod_{k=1}^d (j(\bv)_{k}-j(-\bm{1})_{k}+1)J_k^{-1}}\big|\Big| \DD_n\Big)}\nonumber\\
    &\leq  \sum_{\substack{h_k\geq 0\\1\leq k\leq d}}\E\Big(\max_{\substack{2^{h_k}\leq v_k\leq2^{h_k+1}\\1\leq k\leq d}}\big|\frac{\sum_{[j(-\bm{1}):j(\bv-\bm{1})]} (V_{\bj}- \E(V_{\bj}))/n}{\prod_{k=1}^d (j(\bv-\bm{1})_{k}-j(-\bm{1})_{k}+1)J_k^{-1}}\big|\Big|\DD_n\Big)\nonumber\\
    &\leq  \sum_{\substack{h_k\geq 0\\1\leq k\leq d}}\frac{1}{n\prod_{k=1}^d (r_{n,k}2^{h_k})}\E\Big(\max_{\substack{2^{h_k}\leq v_k\leq2^{h_k+1}\\1\leq k\leq d}}\big|\sum_{[j(-\bm{1}):j(\bv-\bm{1})]} (V_{\bj}- \E(V_{\bj}))\big|\Big|\DD_n\Big)\nonumber\\
    & \leq \sum_{\substack{h_k\geq 0\\1\leq k\leq d}}\frac{1}{n\prod_{k=1}^d (r_{n,k}2^{h_k})}\Big[\E\big(\max_{\substack{2^{h_k}\leq v_k\leq2^{h_k+1}\\1\leq k\leq d}}\big|\sum_{[j(-\bm{1}):j(\bv-\bm{1})]} (V_{\bj}- \E(V_{\bj}))\big|^2\big|\DD_n\big)\Big]^{1/2}\nonumber\\
    & \leq  C_d \sum_{\substack{h_k\geq 0\\1\leq k\leq d}} \frac{1}{n\prod_{k=1}^d (r_{n,k}2^{h_k})}
    \Big[\E\big( \big|\sum_{[j(-\bm{1}):j(2^{\bm{h}+\bm{1}}-\bm{1})]} (V_{\bj} - \E(V_{\bj}))\big|^2\big|\DD_n\big)\Big]^{1/2}\nonumber\\
    &=  C_d \sum_{\substack{h_k\geq 0\\1\leq k\leq d}} \frac{1}{n\prod_{k=1}^d (r_{n,k}2^{h_k})}
    \Big[\sum_{[j(-\bm{1}):j(2^{\bm{h}+\bm{1}}-\bm{1})]}(\alpha_{\bj}+N_{\bj})\Big]^{1/2}
    \label{boundvar}
\end{align}
for some positive constant $C_d$ depending only on $d$. 
As,  by Lemma \ref{lemma:Nj}, $\{\alpha_{\bj}\}$ is bounded uniformly by a constant $a_1>0$ and $N_{\bj} \asymp n/\prod_{k=1}^d J_k$ uniformly for all $\bj$ with $\P_0$-probability tending to one, \eqref{boundvar} is bounded, up to some constant multiple, by
\begin{align*}
     \frac{(n\prod_{k=1}r_{n,k})^{1/2}}{n\prod_{k=1} r_{n,k}}  \sum_{\substack{h_k\geq 0\\1\leq k\leq d}} 2^{-\sum_{k=1}^d h_k/2} =  C \omega_n,
\end{align*}
for some constant $C>0$ with $\P_0$-probability tending to one,
since the sum of the series in the last display converges. Thus \eqref{expratebound} holds.

To bound \eqref{probAn2}, we follow the arguments used in bounding (6.7) by replacing $\bu$ by $\bm{1}$ and lower-bounding  $\bv$ to $\bm{0}$. 

To obtain a bound for \eqref{probBn1}, we proceed as in the proof of Lemma \ref{lemma:AsposP1}, and observe that 
\begin{align*}
    \lefteqn{ \E_0 \max_{\bv\succeq \bm{0}}\Big|\frac{\sum_{[j(-\bm{1}):j(\bv)]} (N_{\bj} - \E_0(N_{\bj}))/n}{\prod_{k=1}^d (j(\bv)_{k}-j(-\bm{1})_{k}+1)J_k^{-1}}\Big| }\\
    &\leq  \sum_{\substack{h_k\geq 0\\1\leq k\leq d}}\E_0 \max_{\substack{2^{h_k}\leq v_k\leq2^{h_k+1}\\1\leq k\leq d}}\Big|\frac{\sum_{[j(-\bm{1}):j(\bv-\bm{1})]} (N_{\bj} - \E_0(N_{\bj}))/n}{\prod_{k=1}^d (j(\bv-\bm{1})_{k}-j(-\bm{1})_{k}+1)J_k^{-1}}\Big| \\
    &\leq  \sum_{\substack{h_k\geq 0\\1\leq k\leq d}} \frac{1}{n\prod_{k=1}^d (r_{n,k}2^{h_k})}
    \E_0  \max_{\substack{2^{h_k}\leq v_k\leq2^{h_k+1}\\1\leq k\leq d}}\Big|\sum_{[j(-\bm{1}):j(\bv-\bm{1})]} (N_{\bj} - \E_0(N_{\bj}))\Big| \\
    &\leq  \sum_{\substack{h_k\geq 0\\1\leq k\leq d}} \frac{1}{n\prod_{k=1}^d (r_{n,k}2^{h_k})}
    \Big[\E_0  \max_{\substack{2^{h_k}\leq v_k\leq2^{h_k+1}\\1\leq k\leq d}}\Big|\sum_{[j(-\bm{1}):j(\bv-\bm{1})]} (N_{\bj} - \E_0(N_{\bj}))\Big|^2 \Big]^{1/2}\\
    &\leq  C_d \sum_{\substack{h_k\geq 0\\1\leq k\leq d}} \frac{1}{n\prod_{k=1}^d (r_{n,k}2^{h_k})}
    \Big[\E_0  \Big|\sum_{[j(-\bm{1}):j(2^{\bm{h}+\bm{1}}-\bm{1})]} (N_{\bj} - \E_0(N_{\bj}))\Big|^2 \Big]^{1/2}\\
    &\leq  C_{d,g_0} \sum_{\substack{h_k\geq 0\\1\leq k\leq d}} \frac{(n\prod_k^d(r_{n,k}2^{h_k+1}))^{1/2}}{n\prod_{k=1}^d (r_{n,k}2^{h_k})},
\end{align*}
which is bounded by a constant multiple of $\omega_n$. 
Thus $\max\{|B_{n,1}(\bm{1},\bv)|:\bv\succeq \bm{0}\} =O_{\P_0}(1)$.

We observe that $B_{n,2}$ is deterministic and 
\begin{align*}
& | B_{n,2}(\bm{1}, \bv)|
= \omega_n^{-1}\Big|\frac{\int_{\bigcup\{I_{\bj}:\bj\in [j(-\bm{1}):j(\bm{0})]\setminus j(\bm{0})\}} g_0(\bx)\d\bx + \int_{I_{j(\bm{0})}} g_0(\bx)\d \bx}{\int_{\bigcup\{I_{\bj}:\bj\in [j(-\bm{1}):j(\bm{0})]\}} \d\bx} - g_0(\bx_0)\Big|,
\end{align*}
As $g_0$ is nonincreasing, it is clear that $g_0(\bx) > g_0(\bx_0)$ for all $\bx \in \bigcup\{I_{\bj}: \bj\in [j(-\bm{1}):j(\bm{0})]\setminus j(\bm{0})\}$. Thus, $| B_{n,2}(\bm{1}, \bv)|$ is bounded by
\begin{align*}
      \omega_n^{-1}\frac{\int_{\bigcup\{I_{\bj}:\bj\in [j(-\bm{1}):j(\bm{0})]\setminus j(\bm{0})\}} g_0(\bx) - g_0(\bx)\d\bx }{\int_{\bigcup\{I_{\bj}:\bj\in [j(-\bm{1}):j(\bm{0})]\}} \d\bx} + \omega_n^{-1} \Big|\frac{\int_{I_{j(\bm{0})}} g_0(\bx) - g_0(\bx_0)\d \bx}{\int_{\bigcup\{I_{\bj}: \bj\in [j(-\bm{1}):j(\bm{0})]\}} \d\bx}\Big|, 
\end{align*}      
which is further bounded by      
\begin{align*}
     \omega_n^{-1}[g_0((\ceil{\bx_0 - \bm{r}_n} - \bm{1}) / \bJ)-g_0(\bx_0)] + o(1)
     \lesssim \big| \sum_{\bmm \in M} \frac{\partial^{\bmm}g_0(\bx_0)}{\bmm!} \big|,
\end{align*}
as $\prod_{k=1}^d r_{n,k}^{m_k} = \omega_n^{\sum_{k=1}^d m_k/\eta_k} = \omega_n$ for $\bmm\in M$. Thus, $\max_{\bv\geq \bm{0}}|B_n(\bm{1}, \bv)| \le M_n/4$ when $n$ large enough.

Piecing these together, \eqref{oneside} follows. 
\end{proof}

\begin{lemma}\label{lemma:dev}
 Let $\bu^{\star}$ and $\bv^{\star}$ be such that
\begin{align*}
    \frac{\sum_{[j(-\bu^{\star}):j(\bv^{\star})]} \theta_{\bj}}{\prod_{k=1}^d (j(\bv^{\star})_{k}-j(-\bu^{\star})_{k}+1)J_k^{-1}}=\min_{\bu\succeq \bm{0}} \max_{\bv\succeq \bm{0}}\frac{\sum_{[j(-\bu):j(\bv)]} \theta_{\bj}}{\prod_{k=1}^d (j(\bv)_{k}-j(-\bu)_{k}+1)J_k^{-1}}.
\end{align*}
Then  there exist $c>1$ and $\gamma>0$ such that
\begin{align*}
    \lim_{c\to \infty}\limsup_{n\to \infty}\Pi(c^{-\gamma} \le \min_{1\leq k \leq d}u^{\star}_k \leq \max_{1\leq k \leq d}u^{\star}_k \le c|\DD_n) = 1,
\end{align*}
in $\P_0$-probability.
\end{lemma}

\begin{proof}
First, we show that 
\begin{equation}
\label{newequation1}
\lim_{c\to\infty}\limsup_{n\to\infty}\Pi(\max_{1\leq k \leq d}u^{\star}_k \le c|\DD_n) = 1 \mbox{ in }\P_0\mbox{-probability}.
\end{equation}
We note that, by the min-max formula,
\begin{multline}
    \omega_n B_{n,2}(\bu^{\star},\bv^{\star})  
    \\
    = \frac{\sum_{[j(-\bu^{\star}):j(\bv^{\star})]} \E(N_{\bj})/n}{\prod_{k=1}^d (j(\bv^{\star})_{k}-j(-\bu^{\star})_{k}+1)J_k^{-1}}- g_0(\bx_0)     \ge \frac{\int_{\bigcup\{I_{\bj}:\bj\in [j(-\bu^{\star}):j(\bm{1})]\}} (g_0(\bx)- g_0(\bx_0))\d\bx}{\int_{\bigcup\{I_{\bj}: \bj\in [j(-\bu^{\star}):j(\bm{1})]\}} \d\bx}.
    \label{formula:llowerbound}
\end{multline}
By Assumption 2, for $\bx$ in a small neighborhood of $\bx_0$, we have
\begin{align*}
    g_0(\bx)- g_0(\bx_0) = \sum_{\bmm\in M}\frac{\partial^{\bmm}g_0(\bx_0)}{\bmm !}(\bx -\bx_0)^{\bmm} + o(\max_k |x_k - x_{0,k}|^{\eta_k}).
\end{align*}
Thus, by noting that $g_0(\bx)$ is monotone nonincreasing and then $\partial^{\eta_k}_k g_0(\bx_0)<0$,
\eqref{formula:llowerbound} can be lower bounded by
\begin{eqnarray*}
  \lefteqn{   \sum_{\bmm\in M}\frac{\partial^{\bmm}g_0(\bx_0)}{\bmm !}
    \frac{\int_{\cup_{[j(-\bu^{\star}):j(\bm{1})]}I_{\bj}} (\bx-\bx_0)^{\bmm} \d \bx }{\int_{\cup_{[j(-\bu^{\star}):j(\bm{1})]}I_{\bj}} \d \bx} + o(\omega_n \max_{k} u^{\star\eta_k}_k)} \\
    &&\asymp  \omega_n \sum_{\bmm\in M}\frac{\partial^{\bmm}g_0(\bx_0)}{(\bmm+\bm{1}) !}
    \frac{1-\prod_{k=1}^d(-u_k^{\star})^{m_k+1}}{\prod_{k=1}^d(1 + u_k^{\star})} + o(\omega_n \max_{k}u^{\star\eta_k}_k)\\
    &&    \gtrsim  \omega_n \max_{k} u_k^{\star\eta_k},
\end{eqnarray*}
which holds when $n$ and $\max_{1\le k \le d} u^{\star}_k$ are sufficiently large.
If $\max_k u^{\star}_k \ge c$ for a sufficiently large $c>0$, then  $|g^{\star}(\bx_0) - g_0(\bx_0)| \gtrsim c\omega_n$. In view of Lemma \ref{lemma:contractionrate}, the assertion in \eqref{newequation1} follows. 

Without loss of generality, it is sufficient to show that $\Pi(u^{\star}_d \leq c^{-\gamma} |\DD_n)\to_{\P_0}0$, as $c,n\to\infty$.
We shall prove the claim by showing that $\{g^{\star}(\bx_0) - g_0(\bx_0) \geq c^{\delta}\omega_n\} $ for some $\delta>0$, to be determined later, if $\{ \min_{1\leq k \leq d} u^{\star}_k \leq c^{-\gamma} \}$ happens when $n$ and $c$ large enough. Then the claim is concluded with the help of Lemma \ref{lemma:contractionrate}.

Recall the notations, $A_{n,1}$, $A_{n,2}$, $B_{n,1}$, $B_{n,2}$, $s_n$, $Y_n$, and $P_n$, defined in the proofs of Theorem 4.1, Lemma \ref{lem:Yn}, and Lemma \ref{lemma:P_nweaklimit}. We shall use the same decomposition as in this proof of Theorem 4.1.
For some constants $a>0$ and $b>0$, to be determined later, we can write 
$\omega_n^{-1}(g^{\star}(\bx_0) - g_0(\bx_0) )$ as 
\begin{align*}
    \lefteqn{ \max_{\bv\succeq \bm{0}}\left(\frac{n}{ \sum_{\bl}V_{\bl} }\cdot\frac{Y_n(\bu^{\star},\bv;V)}{s_n(\bu^{\star},\bv)} + A_{n, 2}(\bu^{\star},\bv;V) + \frac{P_n(\bu^{\star}, \bv)}{s_n(\bu^{\star}, \bv)} + B_{n,2}(\bu^{\star}, \bv)\right)}\\
    &\geq  \max_{ 0 \leq v_k \leq c^a\Ind\{1\leq k \leq d-1\} + c^{-b}\Ind\{k=d\}}
    \Bigg(\frac{n}{ \sum_{\bl}V_{\bl} }\cdot\frac{Y_n(\bu^{\star},\bv;V)}{s_n(\bu^{\star},\bv)} 
    + A_{n,2}(\bu^{\star},\bv;V) + \frac{P_n(\bu^{\star}, \bv)}{s_n(\bu^{\star}, \bv)} + B_{n,2}(\bu^{\star}, \bv)\Bigg)\\
    &\ge  \max_{ 0 \leq v_k \leq c^a\Ind\{1\leq k \leq d-1\} + c^{-b}\Ind\{k=d\}}\left(\frac{n}{ \sum_{\bl}V_{\bl} }\cdot\frac{Y_n(\bu^{\star},\bv;V)}{s_n(\bu^{\star},\bv)} + \frac{P_n(\bu^{\star}, \bv)}{s_n(\bu^{\star}, \bv)} \right) 
    \\
    & \quad +\min_{ 0 \leq v_k \leq c^a\Ind\{1\leq k \leq d-1\} + c^{-b}\Ind\{k=d\}} A_{n,2}(\bu^{\star},\bv;V) \\
    & \quad + \min_{ 0 \leq v_k \leq c^a\Ind\{1\leq k \leq d-1\} + c^{-b}\Ind\{k=d\}} B_{n, 2}(\bu^{\star}, \bv).
\end{align*}

Define $\mathcal{E}_{0}(c)=\{u^{\star}_d<c^{-\gamma}\}$ and $\mathcal{E}_{1}(c)=\{\max_{1\leq k \leq d} u^{\star}_k \leq c\}$. By the first part of the proof, for every $\eta$ and $\epsilon>0$, we have $\P_0(\Pi(\mathcal{E}_1(c)|\DD_n)\geq 1 - \eta)\geq 1-\epsilon$ when $c$ and $n$ are large enough.
Define $R_{a,b,\gamma}(c)=\{(\bu,\bv)\in\RR^d_{\geq 0}\times\RR^d_{\geq 0}: 0\leq u_k \leq c\Ind\{1\leq k \leq d-1\}+c^{-\gamma}\Ind\{k=d\}, 0\leq v_k \leq c^a\Ind\{1\leq k \leq d-1\}+c^{-b}\Ind\{k=d\}\}$, where $a,b,\gamma$ will be determined later.

By Lemma \ref{lem:Yn} and Lemma \ref{lemma:P_nweaklimit}, we know
\begin{align*}
    & Y_n(\bu,\bv; V) \rightsquigarrow \sqrt{g_0(\bx_0)} H_2(\bu, \bv) \text{ in $\P_0$-probability in $\LL_{\infty}([\bm{0},c\bm{1}]\times[\bm{0},c\bm{1}])$},\\
    & P_n(\bu,\bv) \rightsquigarrow \sqrt{g_0(\bx_0)} H_1(\bu, \bv) \text{ in $\LL_{\infty}([\bm{0},c\bm{1}]\times[\bm{0},c\bm{1}])$}.
\end{align*}
Then by Lemma \ref{lemma:max_diff}, when $c,n$ are large enough, there exists a constant $C_1$ depending on $g_0(\bx_0),d,a$ such that $\P_0(\Pi(\mathcal{E}_2(c)|\DD_n)\geq 1-\eta)\geq 1-\epsilon$ where
$\mathcal{E}_2(c)$ is defined as
\begin{align*}
    \mathcal{E}_2(c)=\big\{\sup_{(\bm{u},\bm{v})\in R_{a,b,\gamma}(c)}\abs{Y_{n}(\bm{u},\bm{v})-Y_{n}(\bm{0},\bv)}\leq (C_1/\eta)\sqrt{c^{a(d-1)-\gamma}\log c} \big\}.
\end{align*}
Similarly, there exists $C_2>0$ such that, for 
\begin{align*}
    E_2(c)=\big\{\sup_{(\bm{u},\bm{v})\in R_{a,b,\gamma}(c)}\abs{P_{n}(\bm{u},\bm{v})-P_{n}(\bm{0},\bv)}\leq (C_2/\epsilon)\sqrt{c^{a(d-1)-\gamma}\log c} \big\},
\end{align*}
we have $\P_0(E_2(c))\geq 1 - \epsilon$ when $n$ and $c$ large enough.

As $H_1$ and $H_2$ are two independent Gaussian processes, 
by Lemma \ref{lemma:posimax}, there exists some $\rho_{\eta,\epsilon}>0$ such that, when $a>1$, $c>1$ and $n$ large enough, we have
\begin{align*}
    \lefteqn{\P_0 \times \Pi\Big(\max_{0\leq v_k \leq c^a\Ind\{1\leq k\leq d-1\}+ c^{-b}\Ind\{k=d\}} Y_n(\bm{0},\bv) + P_n(\bm{0},\bv)\leq \sqrt{c^{a(d-1)-b}}\rho_{\eta,\epsilon} \Big)}\\
    &\to  
    \P\Big(\max_{\substack{0\leq v_k \leq c^a\Ind\{1\leq k\leq d-1\}\\ 0\leq v_d \leq c^{-b}}} H_1(\bm{0},\bv) + H_2(\bm{0},\bv) \leq \sqrt{c^{a(d-1)-b}}\rho_{\eta,\epsilon}/\sqrt{g_0(\bx_0)} \Big)\\
    &\leq \P \Big(\max_{\substack{0\leq v_k \leq  1\\1\leq k \leq d}} H_1(\bm{0},\bv) + H_{2}(\bm{0},\bv) \leq \rho_{\eta\epsilon}/\sqrt{g_0(\bx_0)} \Big)\leq \eta\epsilon.
\end{align*}
Hence, there exists a constant $C_3>0$ such that, for the event
\begin{align*}
    \mathcal{E}_3(c)  = \{\sup\{Y_n(\bm{0},\bv)+P_n(\bm{0},\bv): 0\leq v_k \leq c^a, 0\leq v_d\leq c^{-b}\}> C_3\sqrt{c^{a(d-1)-b}} \},
\end{align*} 
we have
$(\P_0\times \Pi)(\mathcal{E}^3_{c}) \geq 1-\eta\epsilon$ for $n$ large enough.
This implies the assertion that 
$\P_0(\Pi(\mathcal{E}_3(c)|\DD_n)\geq1-\eta)\geq1-\epsilon$ when $n$ large enough.


For $s_n$, we have on $R_{a,b,\gamma}(c)$, 
\begin{multline*}
    s_n(\bu,\bv)  = n\omega_n^2\prod_{k=1}^d(\ceil{(x_{0,k} + r_{n,k}v_k)J_k} - \ceil{(x_{0,k} - r_{n,k}u_k)J_k} + 1)J_k^{-1} \\
    \leq \prod_{k=1}^d(u_k + v_k + 2J_k^{-1} r_{n,k}^{-1}) \lesssim (c+c^a)^{d-1}(c^{-b} + c^{-\gamma}),
\end{multline*}
Hence, on the intersection of events $\mathcal{E}_{0}(c),\mathcal{E}_{1}(c),\mathcal{E}_2(c)$, $E_2(c)$ and $\mathcal{E}_{3}(c)$, it holds
that 
\begin{multline*}
    \max_{\substack{0\leq v_k \leq c^a \Ind\{1\leq k\leq d-1\}\\0\leq u_d\leq c^{-b}}} \frac{Y_n(\bm{u}^{\star},\bv;V)+P_n(\bm{u}^{\star},\bv)}{s_n(\bu^{\star},\bv)} \\
    \gtrsim \frac{C_3\sqrt{c^{a(d-1)-b}}-C_1\sqrt{c^{a(d-1)-\gamma}\log c} /\eta- C_2\sqrt{c^{a(d-1)-\gamma}\log c} /\epsilon}{(c^a+c)^{d-1}(c^{-b}+c^{-\gamma})}  \geq C_4 \sqrt{c^{b-a(d-1)}},
\end{multline*}
for some $C_4>0$ when $n$ and $c$ large enough and $\gamma > b$.

To bound $A_{n, 2}(\bu,\bv)$ uniformly in $(\bu,\bv)$, we follow the arguments used in bounding $A_{n, 2}$ in the proof of Theorem 4.1 and Lemma \ref{lemma:Nj}, to observe that 
\begin{align*}
    \frac{\sum_{\bj\in [j(-\bu):j(\bv)]}N_{\bj}/n}{\prod_{k=1}^d (j(\bv)_{k}-j(-\bu)_{k}+1)J_k^{-1}} \asymp \frac{\sum_{\bj\in [j(-\bu):j(\bv)]} \prod_{k=1}^d  J_k^{-1}}{\prod_{k=1}^d (j(\bv)_{k}-j(-\bu)_{k}+1)J_k^{-1}} = 1,
\end{align*}
uniformly for all $(\bu,\bv)$, with $P_0$-probability tending to one.

We also observe that for $n$ sufficiently large, by monotonicity and Assumption 2, there exists $C_5>0$ such that
\begin{align*}
     & \min_{\substack{0\leq v_k \leq c^a \Ind_{1\leq k\leq d-1}\\0\leq v_d\leq c^{-b}}}
    B_{n, 2}(\bu^{\star},\bv) \\
    &=  \omega_n^{-1} \Big(\min_{\substack{0\leq v_k \leq c^a \Ind\{1\leq k\leq d-1\}\\0\leq v_d\leq c^{-b}}} \frac{\int_{\bigcup\{I_{\bj}: \bj\in [j(-\bu^{\star}):j(\bv)]\}} g_0(\bx)\d\bx}{\int_{\bigcup\{I_{\bj}:\bj\in [j(-\bu^{\star}):j(\bv)]\}} \d\bx} - g_0(\bx_0)\Big)\\
    &\geq  \omega_n^{-1} \Big(\min_{\substack{0\leq v_k \leq c^a \Ind\{1\leq k\leq d-1\}\\0\leq v_d\leq c^{-b}}} g_0(\bx_0+\bm{r}_n\bv + \bJ^{-1}) -g_0(\bx_0) \Big)\\
    &=  - \omega_n^{-1} \cdot \max_{\substack{0\leq v_k \leq c^a \Ind\{1\leq k\leq d-1\}\\0\leq v_d\leq c^{-b}}}\sum_{\bmm\in M}\frac{\abs{\partial^{\bmm}g_0(\bx_0)}}{\bmm !}\prod_{k=1}^d(v_kr_{n,k} + J_k^{-1})^{m_k} + o(c^{a\max_{k}\eta_k})\\
     &\geq   -C_5 c^{a\max_k \eta_k}.
\end{align*}

As here we assume that $1\leq \alpha_k < \infty$ for every $k$, we take $a=3$, $b\geq 2 a\max_{k}\alpha_k + a(d-1)$ and $\gamma=b+1$.
 To fulfill the conditions on $a,b$, and $\gamma$ in Lemma \ref{lemma:max_diff}. Let $\delta = (b-a(d-1))/2 \geq a\max_k\eta_k$. 
On the intersection of events $\mathcal{E}_c^{0},\mathcal{E}_c^{1},\mathcal{E}_c^{2}$, $E_c^2$ and $\mathcal{E}^{3}_c$, when $n$ and $c$ large enough, we have
\begin{align*}
    & \omega_n^{-1}(g^{\star}(\bx_0) - g_0(\bx_0) ) \\
    &\geq  \max_{ 0 \leq v_k \leq c^a\Ind\{1\leq k \leq d-1\} + c^{-b}\Ind\{k=d\}}\Big(\frac{n}{ \sum_{\bl}V_{\bl} }\cdot\frac{Y_n(\bu^{\star},\bv;V)}{s_n(\bu^{\star},\bv)} + \frac{P_n(\bu^{\star}, \bv)}{s_n(\bu^{\star}, \bv)} \Big)\\
    & \quad +\min_{ 0 \leq v_k \leq c^a\Ind\{1\leq k \leq d-1\} + c^{-b}\Ind\{k=d\}} A_{n, 2}(\bu^{\star},\bv;V) + \min_{ 0 \leq v_k \leq c^a\Ind\{1\leq k \leq d-1\} + c^{-b}\Ind\{k=d\}} B_{n,2}(\bu^{\star}, \bv)
    \\
    & \gtrsim C_4 \sqrt{c^{b-a(d-1))}}- C_5 c^{a\max_k \eta_k}\gtrsim c^{\delta}.
\end{align*}

Hence for some $C_6>0$, on an event with $\P_0$-probability tending to $1$,
\begin{align}
\label{formula:subound}
    \Pi(\bigcap_{p=0}^3\mathcal{E}_p (c)\cap E_2(c)|\DD_n) \leq  \Pi(\omega_n^{-1}(g^{\star}(\bx_0)-g_0(\bx_0))\geq C_6 c^{\delta}|\DD_n).
\end{align}
By Lemma \ref{lemma:contractionrate}, the right-hand side of \eqref{formula:subound} can be arbitrarily small in $\P_0$-probability we choose $c$ sufficiently large. On the other hand, we know that $\P_0(\Pi(\mathcal{E}_1(c)|\DD_n)\geq 1 - \eta)\geq 1-\epsilon$, $\P_0(\Pi(\mathcal{E}_2(c)|\DD_n)\geq 1 - \eta)\geq 1-\epsilon$, $\P_0(E_2(c))\geq 1-\epsilon$, and $\P_0(\Pi(\mathcal{E}_3(c)|\DD_n)\geq 1 - \eta)\geq 1-\epsilon$. Thus we can conclude that $\P_0(\Pi_n(\mathcal{E}_0(c))\leq \eta)\geq 1-5\epsilon$, when $n$ and $c$ are taken to be large enough.
\end{proof}

\begin{proof}[Proof of Theorem \ref{thm:weaklimit}]

Let $s_n(\bu, \bv) = n\omega_n^{2}\prod_{k=1}^d (j(\bv)_{k}-j(-\bu)_{k}+1)J_k^{-1}$. We claim that $s_n(\bu, \bv) \to \prod_{k=1}^d(u_k + v_k)$ uniformly for $(\bu, \bv) \in [\bm{0},c\bm{1}]\times[\bm{0},c\bm{1}]$ for any $c>0$.
We can bound $s_n$ by
\begin{align*}
n\omega_n^2 \prod_{k=1}^d (u_k r_{n,k} + v_k r_{n,k}) \le s_n(\bu,\bv)
    \le  n\omega_n^2 \prod_{k=1}^d (u_k r_{n,k} + v_k r_{n,k} + 2J_k^{-1}).
\end{align*}
Note that $n\omega_n^2 \prod_{k=1}^n r_{n,k} = 1$. By Lemma \ref{lemma:delta}, we have that 
\begin{align*}
    \left|\prod_{k=1}^d (u_k + v_k + 2(r_{n,k} J_k)^{-1}) - \prod_{k=1}^d (u_k + v_k) \right| \le C_{c,d}  \max_{1\le k \le d} (r_{n,k} J_k)^{-1},
\end{align*}
where $C_{c,d}$ is a positive constant only relative to $c$ and $d$.
Then the claim follows as $J_k\gg r_{n,k}^{-1}$ for all $1\le k\le d$.

As $\sum_{\bl\in[\bm{1}:\bJ]}V_{\bl} \sim \text{Gamma}(\alpha_\cdot + n, 1)$, by Assumption \ref{assumption:priors}, we have 
\begin{align}
\label{sumV}
    \frac{n}{\sum_{\bl\in[\bm{1}:\bJ]}V_{\bl}}-1 = O_{\P}(\max\{n^{-1}\prod_{k=1}^d J_k, n^{-1/2}\}).
\end{align}

We see that $A_{n,1}(\bu, \bv) = Y_n(\bu, \bv; V) \cdot (n/\sum_{\bl} V_{\bl})/s_n(\bu, \bv) $, where $Y_n$ is defined in Lemma \ref{lem:Yn}.
Combining with Lemma \ref{lem:Yn}, we prove that the conditional distribution of $A_{n,1}(\bu,\bv)$ given $\DD_n$ converges weakly to the distribution of $\sqrt{g_0(\bx_0)} H_2(\bu, \bv)/\prod_{k=1}^d(u_k+v_k)$ in $\LL_{\infty}([\bm{0},c\bm{1}]\times[\bm{0},c\bm{1}])$ in $\P_0$-probability.

Secondly, for $A_{n, 2}(\bu, \bv)$, we have
\begin{align}
|A_{n, 2}(\bu, \bv)| &= \omega_n^{-1}\frac{\left|\sum_{[j(-\bu):j(\bv)]} (\alpha_{\bj} + N_{\bj})/\sum_{\bl\in[\bm{1}:\bJ]}V_{\bl} - N_{\bj}/n\right|}{\prod_{k=1}^d (j(\bv)_{k}-j(-\bu)_{k}+1)J_k^{-1}} \nonumber\\
& \le  \frac{a \prod_{k=1}^d J_k}{\omega_n\sum_{\bl\in[\bm{1}:\bJ]}V_{\bl}} + \omega_n^{-1}\left|\frac{n}{\sum_{\bl\in[\bm{1}:\bJ]}V_{\bl}}-1\right|
\times \frac{\sum_{[j(-\bu):j(\bv)]}N_{\bj}/n}{\prod_{k=1}^d (j(\bv)_{k}-j(-\bu)_{k}+1)J_k^{-1}}.
\label{A_n2}
\end{align}
The first term converges in probability to zero in view of \eqref{sumV} and the condition on $J_k$, $\prod_{k=1}^d J_k \ll n\omega_n$. For the second term, by \eqref{sumV}, we can obtain that
\begin{align*}
    \omega_n^{-1}\left|\frac{n}{\sum_{\bl\in[\bm{1}:\bJ]}V_{\bl}}-1\right| = O_{\P}\left(\max\left\{\frac{\prod_{k=1}^d J_k}{n\omega_n}, \frac{1}{ \sqrt{n\omega_n^2} }\right\}\right) \to_{\P} 0,
\end{align*}
in view of the condition $\prod_{k=1}^d J_k \ll n\omega_n$ again and by noting that $\omega_n \gg n^{-1/2}$.
Next, in order to show that $A_{n,2}(\bu, \bv)$ converges to zero in probability uniformly for $(\bu, \bv)\in [c^{-\gamma}, c]^d\times [c^{-\gamma}, c]^d $ for any $c>1$ and $\gamma>1$, 
it suffices to show that 
\begin{align}\label{part2}
    \sup_{\substack{\bu\succeq c^{-\gamma}\bm{1}, \\ \bv\succeq c^{-\gamma}\bm{1}} }\frac{\sum_{[j(-\bu):j(\bv)]}N_{\bj}/n}{\prod_{k=1}^d (j(\bv)_{k}-j(-\bu)_{k}+1)J_k^{-1}} = O_{\P_0}(1).
\end{align}
To this end, Lemma \ref{lemma:AsposP1} establishes that
\begin{align}\label{formula:AsposP1}
    \E_0  \sup_{\substack{\bu\succeq c^{-\gamma}\bm{1}, \\ \bv\succeq c^{-\gamma}\bm{1}} }
    \left|\frac{\sum_{[j(-\bu):j(\bv)]}(N_{\bj}-\E_0(N_{\bj}))/n}{\prod_{k=1}^d (j(\bv)_{k}-j(-\bu)_{k}+1)J_k^{-1}}\right| = O(\omega_n).
\end{align}
Moreover, we observe that
\begin{align*}
    \sup_{\bu\succeq \bm{0}, \bv\succeq\bm{0}}\frac{\sum_{[j(-\bu):j(\bv)]}\E_0(N_{\bj})/n}{\prod_{k=1}^d (j(\bv)_{k}-j(-\bu)_{k}+1)J_k^{-1}} 
    = \sup_{\bu\succeq \bm{0}, \bv\succeq\bm{0}}\frac{\int \Ind_{\bigcup\{I_{\bj}: \bj\in [j(-\bu):j(\bv)]\}} (\bx) g_0(\bx)\d\bx}{\int \Ind_{\bigcup\{I_{\bj}: \bj\in [j(-\bu):j(\bv)]\}}(\bx)\d\bx},
\end{align*}
which is bounded by $g_0(\bm{0})$. Thus, \eqref{part2} follows. We conclude that the posterior probability that $A_{n,2}$ is smaller than a predetermined positive number uniformly for all $(\bu,\bv)\in\RR^d \times \RR^d$ goes to one in $\P_0$-probability.

We write $B_{n, 1}(\bu, \bv) = P_n(\bu, \bv) / s_n(\bu, \bv) $, where $P_n(\bu, \bv)$ is defined in Lemma \ref{lemma:P_nweaklimit} and $s_n$ is defined and investigated in the preceding part of the proof.
Thus $B_{n, 1} \rightsquigarrow \sqrt{g_0(\bx_0)}H_1(\bu, \bv)/\prod_{k=1}^d(u_k + v_k)$ on $\LL_{\infty}([\bm{0}, c\bm{1}]\times [\bm{0},c\bm{1}])$.

We can rewrite $B_{n,2}$ as  
\begin{align}\label{B_n2}
    B_{n, 2}(\bu,\bv)&  = \frac{n\omega_n\int \Ind_{\bigcup\{I_{\bj}: \bj\in [j(-\bu):j(\bv)]\}}(\bx) (g_0(\bx) - g_0(\bx_0)) \d \bx}{s_n(\bu, \bv)}
\end{align}
Under Assumption \ref{assumption:localsmoothness}, by Lemma 1 of  \cite{han2020}, we see that the mixed derivatives with the order $\bmm$ being such that $0<\sum_k m_k/\eta_k < 1$
in the expansion in Assumption \ref{assumption:localsmoothness} must be zero under the multivariate monotonicity condition.
By Assumption \ref{assumption:localsmoothness}, as the remainder in the expansion is $o(\omega_n)$ uniformly for all $(\bu, \bv)\in [\bm{0},c\bm{1}]\times[\bm{0},c\bm{1}]$, the limit of $B_{n,2}(\bu, \bv)$ is the same as that of 
\begin{align*}
    \sum_{\bm{m}\in M} \frac{\partial^{\bmm} g_0(\bx_0)}{\bm{m}!}\frac{n\omega_n \int_{{\bigcup\{I_{\bj}: \bj\in [j(-\bu):j(\bv)]\}}} 
    (\bx-\bx_0)^{\bm{m}}
    \d\bx}{s_n(\bu, \bv)}.
\end{align*}
We note that
\begin{eqnarray*}
    \lefteqn{ \int_{{\bigcup\{I_{\bj}: \bj\in [j(-\bu):j(\bv)]\}}} 
    (\bx-\bx_0)^{\bm{m}}
    \d\bx }\\
    &&= 
    \frac{1}{\prod_{k=1}^d (m_k + 1)}\Big[
    \prod_{k=1}^d\left(\frac{\ceil{(x_{0,k} + v_k r_{n,k})J_k}}{J_k} - x_{0,k}\right)^{m_k+1} \\
    && \quad
    -\prod_{k=1}^d\left(\frac{\ceil{(x_{0,k} - u_k r_{n,k})J_k}-1}{J_k} - x_{0,k}\right)^{m_k+1}
    \Big].
\end{eqnarray*}
As $r_{n,k} = \omega_n^{1/\eta_k}$ and $\sum_{k=1}^d m_k/\eta_k = 1$ for $\bmm \in M$, we have $\prod_{k=1}^d r_{n,k}^{m_k + 1} = \omega_n^{1+\sum_{k=1}^d \eta_k^{-1}}$. Then it follows that
\begin{align*}
 n \omega_n \prod_{k=1}^d\left(\frac{\ceil{(x_{0,k} + v_k r_{n,k})J_k}}{J_k} - x_{0,k}\right)^{m_k+1}
    = \prod_{k=1}^d(v_k + O((J_k r_{n,k})^{-1}))^{m_k + 1} \to \prod_{k=1}^d v_k^{m_k + 1},
\end{align*}
as $J_k \gg r_{n,k}^{-1}$. By the same argument, we have that 
\begin{align*}
    n \omega_n \prod_{k=1}^d\left(\frac{\ceil{(x_{0,k} - u_k r_{n,k})J_k} - 1}{J_k} - x_{0,k}\right)^{m_k+1}
    \to \prod_{k=1}^d (-u_k)^{m_k + 1}.
\end{align*}
In view of \eqref{B_n2}, we have shown that
\begin{align*}
    B_{n, 2}(\bu,\bv) \to \sum_{\bm{m}\in M} \frac{\partial^{\bmm} g_0(\bx_0)}{(\bm{m}+\bm{1})!}\prod_{k=1}^d\frac{v_k^{m_k+1} - (-u_k)^{m_k+1}}{u_k+v_k},
\end{align*}
uniformly for $(\bu, \bv)\in [c^{-\gamma}, c]^d\times [c^{-\gamma}, c]^d $ for any $c>1$ and $\gamma>1$.

The second condition of Proposition \ref{prop:weakconvergence} is shown by Lemma \ref{lemma:dev}. 

For the third condition, from Proposition 7 in \cite{han2020}, we know that 
$$\lim_{c\to \infty} \P(W_c\neq W) = \lim_{c\to \infty} \int \P(W_c \neq W |H_1) \d\P_{H_1} = 0,$$
which implies the third condition by Markov's inequality. 
By verifying all the three conditions in Lemma \ref{prop:weakconvergence}, we conclude the proof of the weak convergence.

\end{proof}

\def\thesection{\Alph{section}}
\setcounter{section}{0} 
\section{Auxiliary results}
\label{sec:appendix}

\begin{lemma}\label{lemma:delta}
For any $c>1$,
let $h:[0,c]^d\times[0,c]^d \mapsto \RR$ be given by $h(\bm{a},\bm{b})=\prod_{k=1}^d(a_k+b_k)$. Then
$|h(\bm{a},\bm{b})-h(\bm{a}',\bm{b}')|\lesssim \|\bm{a}-\bm{a}'\|+\|\bm{b}-\bm{b}'\|$,
where the implicit constant multiple depends only on $c$ and $d$.
\end{lemma}

\begin{proof}
As $h$ is a polynomial function of $2d$ arguments and the order with respect to each argument is one while holding the rest fixed, then, by the first-order Taylor expansion, we have that 
\begin{align*}
    |h(\bm{a}, \bm{b}) - h(\bm{a}', \bm{b}')| 
    \leq \sum_{k=1}^d \frac{\partial h}{\partial a_k}(\bm{a}', \bm{b}')|a_k - a'_k| + \sum_{k=1}^d \frac{\partial h}{\partial b_k}(\bm{a}', \bm{b}')|b_k - b'_k|.
\end{align*}
Note that
$\partial h/\partial a_k(\bm{a}',\bm{b}') = 
\partial h/\partial b_k(\bm{a}',\bm{b}') =
\prod_{l\neq k} (a'_l+b'_l)\leq (2c)^{d-1}$. 
Thus $ |h(\bm{a}, \bm{b}) - h(\bm{a}', \bm{b}')| $  is bounded by a multiple of $\|\bm{a}-\bm{a}'\|_1+\|\bm{b}-\bm{b}'\|_1$.
By the Cauchy–Schwarz inequality, the sum of the $\LL_1$-norms is bounded further by $\sqrt{d}(\|\bu-\bu'\|+\|\bv-\bv'\|)$, which gives the desired inequality.
\end{proof}

\begin{lemma}[Lemma C.6 of  
\cite{han2020}, Supplement C]
\label{lemma:max_diff}
Let $a\in(1,\infty)$, $b\in(0, \gamma)$, and $\gamma\in(b, a + b - 1)$. Let $R_{a,b,\gamma}(c) = \{ (\bu,\bv)\in \RR^d_{\geq 0}\times \RR^d_{\geq 0} : u_k\in[0,c], 1\leq k \leq d; u_d \leq c^{-\gamma}; 0\leq v_k \leq c^a, 1\leq k \leq d; 0\leq v_d \leq c^{-b}\}$. Let $H(\bu,\bv)$ be a centered Gaussian process with covariance kernel 
\begin{align}
    \mathrm{Cov}(H(\bu,\bv), H(\bu',\bv')) = \prod_{k=1}^d(u_k \wedge u_k' +v_k \wedge v_k' ), 
    \label{covariance H}
\end{align}
for all $(\bu,\bv),(\bu',\bv') \in \RR^d_{\geq0}\times \RR^d_{\geq 0}$. 
Then there exists a constant $C_{a,d}>0$ such that for any $c>1$,
\begin{align*}
    \E\big( \sup_{(\bu,\bv)\in R_{a,b,\gamma}(c)} |H(\bu, \bv) - H(\bm{0},\bv)|\big) \leq C_{a,d} \sqrt{c^{a(d-1)-\gamma}\log c}.
\end{align*}
\end{lemma}

\begin{lemma}[Lemma B.4 of \cite{Kang_Coverage}]
\label{lemma:Nj}
If $q_1 \leq g_0 \leq q_2$ for $0<q_1\le q_2<\infty$, then
\begin{align*}
    \P_0(q_1n/(2\prod J_k) \leq \min_{\bj} N_{\bj} \leq \max_{\bj} N_{\bj} \leq 2q_2n /(\prod J_k) )\to 1.
\end{align*}
 
\end{lemma}

\begin{lemma}\label{lemma:posimax}
    For the Gaussian process $H(\bu,\bv)$ with covariance kernel  \eqref{covariance H}, 
            $$ \max\{H(\bm{0},\bv): 0\leq v_k \leq  1, \, 1\leq k \leq d\}  \geq 0 \text{ a.s.}$$
  
\end{lemma}

\begin{proof}
   For $d=1$, $H(0, v)$ is a Brownian motion on $v\ge 0$. The first statement is simply induced by the reflection principle. When $s=d$, by noting that $H(\bm{0}, (1,\ldots, 1, v_d))$ is a Brownian motion with respect to $v_d>0$, it holds that
\begin{align*}
    0 \leq \P \Big(\max_{\substack{0\leq v_k \leq  1\\1\leq k \leq d}} H(\bm{0},\bv) \leq 0 \Big) \leq \P\Big(\max_{0\leq v_d \leq 1} H(\bm{0}, (1,\ldots, 1, v_d)) \leq 0 \Big) = 0
\end{align*}
by the fact that 
$
    \P\Big(\max_{0\leq v_d \leq 1} H(\bm{0}, (1,\ldots, 1, v_d)) \leq \rho \Big) = 
    1-2\P(H(\bm{0}, \bm{1}) \ge \rho ),
$
and the continuity of normal distribution.
\end{proof}

\begin{lemma}
\label{lemma:gammamoments}
If $U\sim \mathrm{Gamma}(\Delta, 1)$, then for every $m\in \ZZ_{>0}$, 
    $ \E( U - \E(U) )^{2m} \lesssim \Delta^m$  as  $ \Delta\to\infty$.
\end{lemma}

\begin{proof}
Suppose $n_{\Delta} \le \Delta < n_{\Delta}+1$ for some integer $n_{\Delta}$. We will assume that $\Delta>n_{\Delta}$ in this proof. If $\Delta$ is an integer, the argument can proceed in a similar and simpler way.
Let $Z_i$ denote independent random variables with standard exponential distribution for $i = 1,\ldots, n_{\Delta}$. Let $Z_{n_{\Delta} + 1}\sim \text{Gamma}(\Delta-n_{\Delta}, 1)$. Then $U$ is equal to $\sum_{i=1}^{n_{\Delta} + 1} Z_i$ in distribution. By Marcinkiewicz–Zygmund inequality, there exists a constant $C_m>0$ depending only on $m$ such that,
\begin{align}
\label{inequ:2mmoment}
    \E(U - \E(U) )^{2m} = \E\big[ \sum_{i=1}^{n_{\Delta}+1} (Z_i - \E(Z_i)) \big]^{2m} 
    \leq C_m  \E\big[ \sum_{i=1}^{n_{\Delta}+1} (Z_i - \E(Z_i))^2 \big]^{m}.
\end{align}
By Jensen's inequality,
\begin{align*}
    \big[\frac{1}{n_{\Delta} + 1} \sum_{i=1}^{n_{\Delta}+1} (Z_i - \E(Z_i))^2 \big]^{m} \le \frac{1}{n_{\Delta} + 1}  \sum_{i=1}^{n_{\Delta}+1} (Z_i - \E(Z_i))^{2m}.
\end{align*}
The the right-hand side of \eqref{inequ:2mmoment} is bounded by $(n_{\Delta}+1)^{m-1}\sum_{i=1}^{n_{\Delta}+1} \E(Z_i - \E(Z_i))^{2m}$. For fixed $m$, the lemma follows when $\Delta \to \infty$.
\end{proof}

\section*{Acknowledgments}

This research is partially supported by NSF grant number DMS-1916419.

\bibliographystyle{plain}
\bibliography{mybib}

\end{document}